\documentclass{amsart}
\usepackage{marktext}
\usepackage{gimac}
\renewcommand{\le}{\leqslant}
\renewcommand{\ge}{\geqslant}

\usepackage{accents}
\newcommand{\ubar}[1]{\underaccent{\bar}{#1}}

\newcommand{\GI}[2][]{\sidenote[colback=yellow!20]{\textbf{GI\xspace #1:} #2}}
\newcommand{\RP}[2][]{\sidenote[colback=green!10]{\textbf{RP\xspace #1:} #2}}
\newcommand{\HL}[2][]{\sidenote[colback=orange!20]{\textbf{HL\xspace #1:} #2}}

\RequirePackage{datetime,currfile}

\makeatletter
\newcommand{\leqnomode}{\tagsleft@true\let\veqno\@@leqno}
\newcommand{\reqnomode}{\tagsleft@false\let\veqno\@@eqno}
\newcommand{\mylabel}[2]{\def\@currentlabel{#2}\label{#1}}
\makeatother

\newcommand{\D}{\partial}
\newcommand{\Dt}{\partial_t}
\newcommand{\Dx}{\partial_x}
\newcommand{\Dxx}{\partial_x^2}

\newcommand{\loc}{_{\textup{loc}}}

\newcommand{\eps}{\epsilon}

\newcommand{\sgn}{\operatorname{sign}}

\def\Xint#1{\mathchoice
  {\XXint\displaystyle\textstyle{#1}}%
  {\XXint\textstyle\scriptstyle{#1}}%
  {\XXint\scriptstyle\scriptscriptstyle{#1}}%
  {\XXint\scriptscriptstyle\scriptscriptstyle{#1}}%
\!\int}
\def\XXint#1#2#3{ \setbox0=\hbox{$#1{#2#3}{\int}$ }
\vcenter{\hbox{$#2#3$ }}\kern-.6\wd0}
\def\avint{\Xint-}

\newcommand{\lin}{\mathit{lin}}
\newcommand{\BV}{\mathit{BV}}

\begin{document}

\title[Global dynamics and Photon Loss in the Kompaneets Equation]
      {Global Dynamics and Photon Loss in the Kompaneets Equation}

\date{\today} 

\author[Ballew]{Joshua Ballew}
\address{
  Department of Mathematics and Statistics,
  Slippery Rock University,
  Slippery Rock, PA 16057, USA.
}
\email{joshua.ballew@sru.edu}

\author[Iyer]{Gautam Iyer}
\address{%
  Department of Mathematical Sciences,
  Carnegie Mellon University,
  Pittsburgh, PA 15213.}
\email{gautam@math.cmu.edu}

\author[Levermore]{C. David Levermore}
\address{
  Department of Mathematics,
  University of Maryland,
  College Park, MD 20742-4015.}
\email{lvrmr@umd.edu}

\author[Liu]{Hailiang Liu}
\address{%
   Department of Mathematics,
   Iowa State University,
   Ames, IA 50011.}
\email{hliu@iastate.edu}
\author[Pego]{Robert L. Pego}
\address{%
  Department of Mathematical Sciences,
  Carnegie Mellon University,
  Pittsburgh, PA 15213.}
\email{rpego@cmu.edu}

\thanks{%
  This work has been partially supported by
  the National Science Foundation under
  the NSF Research Network Grant nos.\ RNMS11-07444, \ RNMS11-07291(KI-Net),
  grants
  NSF DMS-1814147 and DMS-2108080 to GI, 
  NSF Grants DMS07-57227 and DMS09-07963 to HL,
  DMS 2106534 (RLP),
  and the Center for Nonlinear Analysis.
}

\begin{abstract} 
  The Kompaneets equation governs dynamics of the photon energy spectrum in certain high temperature (or low density) plasmas.
  We prove several results concerning the long-time convergence of solutions to Bose--Einstein equilibria 
  and the failure of photon conservation.
  In particular, we show the total photon number can decrease with time  
  via an outflux of photons at the zero-energy boundary. 
  The ensuing accumulation of photons at zero energy is analogous to Bose--Einstein condensation.
  We provide two conditions that guarantee that photon loss occurs, and show that once loss is initiated then it persists forever.
  We prove that as $t\to \infty$, solutions necessarily converge to equilibrium and we characterize the limit in terms of the total photon loss.
  Additionally, we provide a few results concerning the behavior of the solution near the zero-energy boundary, an Oleinik inequality, a comparison principle, and show that the solution operator is a contraction in $L^1$.
  None of these results impose a boundary condition at the zero-energy boundary.
\end{abstract}
\maketitle

\section{Introduction.}

The Kompaneets equation governs evolution of the photon energy spectrum in high temperature (or low density) plasmas, which are spatially uniform, isotropic, isothermal and non-relativistic, and in which the dominant energy exchange mechanism is Compton scattering.
This equation was first derived by Kompaneets~\cite{Kompaneets57}
and
is now fundamental to modern cosmology and high-energy astrophysics.
It has applications in the study of the interaction between matter and radiation in the early universe, the radiation spectra for the accretion disk around black holes, and the Sunyaev--Zeldovich effect, the reduction of the cosmic microwave background (CMB) brightness in near clusters of galaxies with hot gases~\cite{SunyaevZeldovich70,SunyaevZeldovich72,ShakuraSunyaev73,Birkinshaw99}.

Mathematically, the Kompaneets equation takes the non-dimensional form
\begin{equation}\label{e:kompaneetsF}
  \partial_t f 
  = \frac{1}{x^2} \, 
    \partial_x \brak[\big]{ x^4 \paren[\big]{ \partial_x f + f + f^2 } } \,,
    \qquad x \in (0, \infty)\,,
    ~t > 0\,.
\end{equation}
Here $x$ is proportional to photon energy and $t$ is proportional to time.
The variable $f$ expresses the photon number density relative to the 
measure $x^2dx$ that appears due to the assumption of isotropy in a
3-dimensional space of wave vectors.
The physically relevant boundary condition at infinity requires that the incoming photon flux vanishes.
The boundary at $x = 0$ requires more care to understand, as the diffusion coefficient vanishes.
Escobedo et al.~\cite{EscobedoHerreroEA98} showed that solutions to~\eqref{e:kompaneetsF} are (globally) unique without imposing any boundary condition at $x = 0$.
One interesting feature of the Kompaneets equation is that the absence of the boundary condition at $x = 0$ allows for photon loss,
despite the fact that photon numbers are nominally conserved in Compton scattering.
In this paper we provide rigorous results describing the manner by which photons can be lost through an outflux at $x = 0$.

Physically, an outflux of photons at $x = 0$ means that a macroscopic number of photons accumulate at negligible values of energy. 
This phenomenon has been regarded by several authors~\cite{ZelDovichLevich69,Syunyaev71,CaflischLevermore86,EscobedoHerreroEA98,EscobedoMischler01,JosserandPomeauEA06,KhatriSunyaevEA12,LevermoreLiuEA16} as analogous to the formation of a Bose--Einstein condensate---a collection of bosons occupying the same minimum-energy quantum state.
While the existence of such condensates was predicted in 1924 by Bose and Einstein~\cites{Bose24,Einstein25,Einstein25a}, they were first exhibited in 1995~\cite{AndersonEnsherEA95} for Rubidium-87 vapor.
For photons, Bose--Einstein condensates were experimentally exhibited in 2010~\cite{KlaersSchmittEA10}, but in circumstances dominated by physics different from Compton scattering.
In the cosmological setting, true Bose--Einstein condensation is thought to be suppressed by other physical mechanisms 
that become important at very low energy \cite{KhatriSunyaevEA12}.

To briefly explain the outflux phenomenon for \eqref{e:kompaneetsF}, we note that the total photon number is
\begin{equation*}
  \mathcal N(f_t) \defeq \int_0^\infty f_t(x)\,x^2 dx\,.
\end{equation*}
(We clarify that $f_t(x)$ here is the value of $f$ at $(t, x)$, and not the time derivative.)
By multiplying~\eqref{e:kompaneetsF} by $x^2$ and integrating, one immediately sees that the total photon number is a conserved quantity, provided the photon flux vanishes at both $0$ and $\infty$.

Now it is also known (see for instance~\cites{CaflischLevermore86,LevermoreLiuEA16}, or Section~\ref{s:largetime}, below) that solutions to~\eqref{e:kompaneetsF} formally dissipate a quantum entropy.
This suggests that $f_t$ converges to an equilibrium solution as $t \to \infty$.
The nonnegative equilibrium solutions to~\eqref{e:kompaneetsF} can readily be computed by solving an ODE, and are given by Bose--Einstein statistics, taking the form
\begin{equation}\label{e:fhat}
  \hat f_\mu(x) \defeq \frac{1}{e^{x + \mu} - 1}\,,
  \qquad\text{for } \mu \geq 0\,.
\end{equation}
(For mathematical convenience, as in \cite{EscobedoHerreroEA98} we take the parameter $\mu$ as proportional to the negative of chemical potential.)
Of course, one can now compute the maximum photon number in equilibrium to be
\begin{equation*}
  \sup_{\mu \geq 0} \mathcal N( \hat f_\mu)
    = \mathcal N(\hat f_0)
    = \int_0^\infty \frac{x^2}{e^x - 1} \, dx
    = 2\zeta(3)
    \approx 2.404\dots
    < \infty \,.
\end{equation*}
Here $\zeta(s) = \sum_1^\infty 1/k^s$ is the Riemann zeta function.
Thus if one starts~\eqref{e:kompaneetsF} with initial data such that $\mathcal N(f_0) > \mathcal N( \hat f_0 )$, then the total photon number can not be a conserved quantity---at least not in the infinite-time limit---and there must be a dissipation mechanism through which photons are lost.


In this paper we prove that photons are indeed lost in finite time through an outflux at $0$ 
for all solutions of the Kompaneets equation \eqref{e:kompaneetsF} with $\mathcal N(f_0) > \mathcal N( \hat f_0 )$. 
However, $\mathcal N(f_0) \leq \mathcal N(\hat f_0)$ does not guarantee that the photon number is conserved, and a family of examples was constructed by Escobedo et\ al.~\cite{EscobedoHerreroEA98}.
Here we provide an explicit condition on the initial data guaranteeing a finite time photon outflux, that does not necessitate or preclude $\mathcal N(f_0) > \mathcal N(\hat f_0)$.
We also prove several other results concerning existence, uniqueness, and 
convergence to equilibrium in the long time limit. 

To state our results it is convenient to reformulate the problem in terms of the photon number density with respect to the measure $dx$, defined by
\begin{equation}\label{e:nDef}
  n_t(x) \defeq x^2 f_t(x)\,.
\end{equation}
In terms of this photon number density, the stationary solutions~\eqref{e:fhat} are now
\begin{equation}\label{e:nhat}
  \hat n_\mu(x) \defeq \frac{x^2}{e^{x + \mu} - 1} \,,
  \qquad
  \text{for } \mu \geq 0\,,
\end{equation}
and the total photon number is now
\begin{equation*}
  N(n_t) \defeq \int_0^\infty n_t(x) \, dx
    = \int_0^\infty x^2 f_t(x) \, dx
    = \mathcal N(f_t)\,.
\end{equation*}
Two of our main results can be stated non-technically as follows:
\begin{enumerate}\reqnomode
  \item
    We show that the total photon number is non-increasing in time, and can only decrease through an outflux of photons at $x = 0$.
    The boundary conditions ensure that photons can never be lost to (or gained from) infinity, and so the fact that the total photon number is decreasing means that there can never be an influx of photons at $x = 0$.
    Moreover, for $0 \leq s \leq t$, we prove the following loss formula for the total photon number:
    \begin{equation}\label{e:lossIntro}
      N(n_t) = N(n_s) - \int_s^t n_\tau(0)^2 \, d\tau \,.
    \end{equation}

  \item
    If the initial data is not identically~$0$, then we prove that as $t \to \infty$, the solution converges strongly in~$L^1$ to one of the equilibrium solutions~$\hat n_\mu$.
    The parameter~$\mu \in [0, \infty)$ is characterized by the property
    \begin{equation*}
      N(\hat n_\mu) = N(n_0) - \int_0^\infty n_t(0)^2 \, dt\,.
    \end{equation*}
    In particular, $N(n_t)\to N(\hat n_\mu)$ as $t\to\infty$. Consequently, if $N(n_0)>N(\hat n_0)$, photon loss must occur in finite time.
\end{enumerate}

Even though we show there exists $\mu \in [0, \infty)$ such that $n_t \to \hat n_\mu$ as~$t \to \infty$, there appears to be no general way to determine~$\mu$ from the initial data.
There are however two cases where this can be done:
\begin{enumerate}
  \item
    If $n_0 \geq \hat n_0$, then $\mu = 0$ and $n_t \to \hat n_0$ in $L^1$ as $t \to \infty$.
  \item
    If $n_0 \leq \hat n_0$ on the other hand, then there is no photon loss, and $\mu \geq 0$ will be the unique number such that $N(n_0) = N(\hat n_\mu)$.
\end{enumerate}
\medskip

In light of~\eqref{e:lossIntro}, we see the change in the total photon number between time $0$ and $t$ is exactly $\int_0^t n_s(0)^2 \, ds$. 
Following precedent, one can interpret the accumulated outflux at zero energy, as the ``mass'' of a Bose--Einstein condensate at time $t$.
Clearly this is non-decreasing as a function of time.
We will in fact show a stronger result: once photon loss is initiated, it persists for all time without stopping.
That is,
there exists $t_* \in [0, \infty]$ such that $n_t(0) = 0$ for all $t < t_*$ and $n_t(0) > 0$ for all $t > t_*$.
This time $t_*$ can be viewed as the time the Bose--Einstein condensate starts forming.

There are situations in which photon loss never occurs 
(i.e., $t^* = +\infty$).
Indeed, we show that if $n_0 \leq \hat n_0$, then total photon number is conserved 
and $t_*$ is infinite.
On the other hand, there are certain scenarios under which one can prove $t_*$ is finite.
One such scenario was previously identified by Escobedo et\ al.~\cite{EscobedoHerreroEA98} where the authors construct a family of solutions that develop a Burgers-like shock at $x = 0$ in finite time.
In this paper we prove that $t_*$ is finite in two different scenarios:
\begin{enumerate}
  \item 
    If the total photon number initially is larger than $N(\hat n_0) = 2\zeta(3)$, the maximum photon number in equilibrium, then $t_*$ must be finite.

  \item
    If $\partial_x n_0(0) > 1$, then we show $t_*$ is finite, and furthermore we provide an explicit upper bound for $t_*$.
\end{enumerate}

To the best of our knowledge the second scenario above was not identified earlier.
The first scenario mathematically rules out the possibility that the Kompaneets equation
allows photons to remain conserved while concentrating at small but positive energy. 
Such behavior was suggested in \cite{CaflischLevermore86} by number-conserving numerical simulations 
and was shown to be compatible with entropy-minimization arguments.
Our results on photon loss indicate, instead, that numerical schemes for the Kompaneets equation 
should \emph{not} be designed to conserve photon number at the zero-energy boundary,
since solutions of \eqref{e:kompaneetsF} do not have this property in general.

Analogs of many of our current results were previously obtained by the authors
in~\cite{LevermoreLiuEA16} and~\cite{BallewIyerEA16} for different simplified
models of~\eqref{e:kompaneetsF} obtained by neglecting terms that seem inessential to the photon loss phenomenon
but fail to preserve true Bose--Einstein equilibria.
As in those previous works, we make essential use of mathematical tools traditionally associated with first-order nonlinear
conservation laws without diffusion, such as an $L^1$ contraction property for solutions, negative slope bounds (an Oleinik inequality),
and comparisons to compression and rarefaction waves.
The proofs for the full Kompaneets equation~\eqref{e:kompaneetsF}, however, are significantly different and more involved.
All of our main results (including precise statements of those non-technically described earlier) are stated in Section~\ref{s:results} below.
In a sense, our results justify the notion, examined nonrigorously by many previous authors,
that at the zero-energy boundary the diffusion term in \eqref{e:kompaneetsF} can be neglected 
and the flux is dominated by the nonlinear advection term $n^2 = x^4 f^2$ which arises 
from a quantum enhancement of scattering into states occupied by bosons.


There are various mathematical or physical mechanisms that may prevent (or permit)
loss of photons at zero energy or formation of a true Bose--Einstein condensate. 
Kompaneets derived \eqref{e:kompaneetsF} in a Fokker--Planck approximation 
to a quantum Boltzmann equation for photon scattering from a fixed
Maxwellian electron distribution.
This equation can be written in the nondimensional form 
\begin{align*}
 \partial_t f(x,t) &= 
\int_0^\infty \Bigl(
f(x_*,t)\bigl(1+f(x,t)\bigr)\,e^{-x}\\ 
& \qquad\qquad -f(x,t)\bigl(1+f(x_*,t)\bigr)\,e^{-x_*}\Bigr) \sigma(x,x_*)x_*^2\,dx_*\,,
\end{align*}
cf.~\cite[Eq.~(12.47)]{Castor04}, \cite[Eq.~(5.67)]{EscobedoMischlerEA03} and \cite{CortesEscobedo19}.
In this Boltzmann--Compton equation the form of $\sigma(x,x_*)=\sigma(x_*,x)$
is determined by approximating the Klein--Nishina cross section for Compton scattering,
and is strongly peaked where $x\approx x_*$.
Under a simplifying boundedness assumption on the scattering kernel,
Escobedo \& Mischler \cite{EscobedoMischler01} showed that photon loss in
finite time is impossible, but a concentration of photons approaching zero
energy appears in the limit $t\to\infty$.
The behavior in such a case was further studied in \cite{EscobedoMischlerEA04}. 
For the physical kernel, Ferrari \& Nouri \cite{FerrariNouri06}
showed that if the initial data everywhere exceeds the Planckian 
($f_0\ge \hat f_0$), then a number-conserving weak solution fails to exist 
for any positive time, while if $f_0\le \hat f_0$ then a global 
solution exists.
Recent work by Cort\'es \& Escobedo \cite{CortesEscobedo19} reviews related results
and revisits the question with a different kind of kernel truncation, obtaining
an existence result that does not preclude formation of a Dirac delta mass at zero energy.

Physical effects that become important at low energy and destroy photon
conservation include Bremsstrahlung and double Compton scattering \cite{KhatriSunyaevEA12}.
The derivation of the Kompaneets equation from a quantum Boltzmann equation
has also been revisited recently by several authors in the physical literature \cite{OliveiraMaesEA21,MendoncaTercas17,Milonni21}.
E.g., Mendon{\c{c}}a \& Ter{\c{c}}as \cite{MendoncaTercas17} suggest that photons in some plasmas 
can have an effective mass at zero frequency, and modify the Kompaneets equation accordingly.

Finally, we mention that a considerable body of work exists concerning the 
related but distinct phenomenon 
of Bose--Einstein condensation in quantum models of boson-boson scattering,
described by Boltzmann--Nordheim (or Uehling--Uhlenbeck) equations. 
For analytical studies of condensation phenomena and convergence to equilibrium
in these models we refer to 
\cite{Spohn10,EscobedoVelazquez15,Lu13,LuMouhot12,LuMouhot15,Lu18,CaiLu19},
the book \cite{PomeauTran19}, and references therein.
For the Boltzmann--Nordheim equation, Fokker--Planck-type approximations of higher order 
have been developed formally by Josserand et al.~\cite{JosserandPomeauEA06}
and analyzed in work of J\"ungel \& Winkler \cite{ JungelWinkler15, JungelWinkler15a}.
There are also studies of other nonlinear Fokker--Planck models 
that admit Bose--Einstein equilibria and the possibility of condensation,
and we refer the reader to \cite{Toscani12,CarrilloDiFrancescoEA16,CarrilloHopfEA20}
for work on this and further references.

Even though we do not consider alternative models or mechanisms here,
our present results for solutions of the Kompaneets equation itself 
provide some understanding of the mechanism by which Compton scattering creates a photon flux towards low energy.

\section{Main Results.}\label{s:results}

This section is devoted to stating precise versions of our main results.
In terms of the photon number density~$n$ (defined in~\eqref{e:nDef}), equation~\eqref{e:kompaneetsF} becomes
\begin{equation}\label{e:komp}
  \partial_t n = \partial_x J\,,
  \qquad
  \text{where}\quad
  J = J(x, n) \defeq x^2 \partial_x n + (x^2 - 2x) n + n^2\,,
\end{equation}
is the photon flux to the left.
We will study equation~\eqref{e:komp} with bounded nonnegative initial data $n_0$, and impose the no-flux boundary condition
\begin{equation}\label{e:noflux}
  \lim_{x \to \infty} J(x, n) = 0
\end{equation}
at infinity.
As mentioned earlier, we will not impose any boundary condition at $x = 0$.

\subsection{Construction and Properties of Solutions.}

Previous work of Escobedo et\ al.~\cite{EscobedoHerreroEA98} shows there is a \emph{unique}, globally regular solution to~\eqref{e:komp}--\eqref{e:noflux}.

\begin{theorem}[Existence~\cite{EscobedoHerreroEA98}]\label{t:exist}
  For any bounded measurable $n_0 \geq 0$ satisfying
  \begin{equation}\label{e:x2n}
    \lim_{x \to \infty} x^2 n_0(x) \to 0\,,
  \end{equation}
  there exists a unique nonnegative function
  \begin{equation}\label{e:nSpace}
    n \in C( [0, \infty); L^1) \cap L^\infty_{\loc}( [0, \infty); L^\infty( [0, \infty) ) \cap  C^{2,1}((0, \infty)^2)
  \end{equation}
  that is a solution to \eqref{e:komp}--\eqref{e:noflux} with initial data~$n_0$.
  Moreover, for any~$T < \infty$ we have
  \begin{gather}\label{e:x2nPtwise}
    \lim_{x\to \infty}x^2 n_t(x) =0 \,,
    \qquad\text{uniformly for } 0 \leq t < T\,,
    \\
    \label{e:x2dxnVanish}
    \lim_{x\to \infty} x^2 \partial_x n(x, t) =0\,,
      \qquad\text{uniformly on the set } \frac{1}{x^2} \leq t < T\,.
  \end{gather}
\end{theorem}

Existence and uniqueness of the solution is the content of Theorem 2 in~\cite{EscobedoHerreroEA98}.
The vanishing conditions~\eqref{e:x2nPtwise}--\eqref{e:x2dxnVanish} are contained in the proof of this theorem on page~3850, and we refer the reader to~\cite{EscobedoHerreroEA98} for details.
\begin{remark}\label{r:bounded}
  In Section~\ref{s:pointwise} (Proposition~\ref{p:nBdd}) we will show later that $n$ is bounded globally in both space and time, not just locally bounded as stated in~\eqref{e:nSpace}, and moreover infer from \cite{EscobedoHerreroEA98} that \eqref{e:x2nPtwise} and \eqref{e:x2dxnVanish} hold with $T=\infty$.
\end{remark}
\begin{remark}
  In~\cite{EscobedoHerreroEA98}, the authors assumed the initial data~$n_0$ is continuous.
  This assumption isn't necessary and can be relaxed using a density argument and the $L^1$ contraction property 
  (cf.~\cite{LevermoreLiuEA16}).
\end{remark}
\begin{remark}
  We reiterate that while Theorem~\ref{t:exist} requires a no-flux boundary condition at $x = \infty$ (equation~\eqref{e:noflux}), it does not require a boundary condition at $x = 0$.
  That is, Theorem~\ref{t:exist} guarantees that solutions to~\eqref{e:komp}--\eqref{e:noflux} are globally unique, without requiring any boundary condition at $x = 0$.
\end{remark}

From the vanishing conditions~\eqref{e:x2nPtwise}--\eqref{e:x2dxnVanish} we immediately see that
\begin{equation}\label{jr}
  \lim_{R\to \infty}\int_s^t \abs{J(R, n_\tau)} \, d\tau = 0 \,,
  \qquad\text{for any } 0 < s \leq t \,.
\end{equation}
The behavior of the solution at $0$, however, is a little more delicate.
The constructed examples in~\cite{EscobedoHerreroEA98} (discussed below) show that the function $n$ can not always be extended continuously to a function defined on the domain $(x, t) \in [0, \infty) \times (0, \infty)$.
However, for every $t > 0$ the function $n_t(x)$ can be extend continuously at $x = 0$.
\begin{lemma}[Continuity at $x = 0$]\label{l:n0exist}
  Let $T > 0$, let $Q_T = (0, \infty) \times (0, T)$ and $n \in L^\infty(Q_T)$ be a nonnegative solution of~\eqref{e:komp}--\eqref{e:noflux} whose initial data satisfies~\eqref{e:x2n}.
  Then for any $t \in (0, T]$,  the limit of $n_t(x)$ exists as $x \to 0^+$.
\end{lemma}

We will subsequently use the notation $n_t(0)$ to denote $\lim_{x \to 0^+} n_t(x)$ for every $t > 0$.
The proof of Lemma~\ref{l:n0exist} is presented in Section~\ref{s:construct}.
The key ingredient in the proof is an Oleinik inequality (Lemma~\ref{l:oleinik}, below) establishing explicit negative slope bounds on solutions.
{Such inequalities typically arise in the study of hyperbolic problems.
Even though~\eqref{e:komp} is parabolic, the degeneracy near $x = 0$ makes the system exhibit an ``almost hyperbolic'' behavior, and hence is amenable to analysis using such tools.}

Next, we study convergence of the flux $J(x, n)$ as $n \to 0$.
As we will shortly see, the slope $\partial_x n$ may develop a singularity at $x = 0$.
We don't presently know the exact singularity profile, and hence do not know whether or not for every $t > 0$ we have $x^2 \partial_x n_t(x) \to 0$ as $x \to 0$.
As a result, we don't know whether or not the flux satisfies $\lim_{x \to 0} J(x, n_t) = n_t(0)^2$ for every $t  > 0$.
We claim, however, that we do have the time integrated version $\lim_{x \to 0} \int_s^t J(x, n_\tau) \, d\tau \to \int_s^t n_\tau(0)^2$, and this is our next lemma.

\begin{lemma}[Flux behavior at $x = 0$]\label{l:intJ0}
  For any $0 \leq s \leq t$ we have
  \begin{equation}\label{e:intJ0}
    \lim_{x\to 0^+} \int_s^t J(x,\tau)\,d\tau = \int_s^t n^2(0,\tau)\,d\tau. 
  \end{equation}
  If, further, $0 < s \leq t$, then we have the stronger convergence
  \begin{align}\label{e:a0}
    \lim_{x \to 0} \int_s^t x^2 \abs{ \partial_x n_\tau(x)} \, d\tau =  0 \,,
    \quad\text{and}\quad
    \lim_{x \to 0} \int_s^t \abs{ J(x, n_\tau) -n_\tau^2(0) } \, d\tau = 0 \,.
  \end{align}
\end{lemma}

We are presently unaware whether or not~\eqref{e:a0} holds if $s = 0$.
The proof of Lemma~\ref{l:intJ0} is somewhat indirect.
We first establish a loss formula (Proposition~\ref{p:lossFormula}, below) relating the decrease in total photon number to the outflux of zero-energy photons.
(For clarity of presentation we state Proposition~\ref{p:lossFormula} in next subsection, as it fits better with our results on condensate formation.)
It turns out that the loss formula can be used to prove convergence of the time integrated flux as stated in~\eqref{e:intJ0}.
To obtain the stronger convergence stated in~\eqref{e:a0}, we require the Oleinik inequality (Lemma~\ref{l:oleinik}), and hence require $s > 0$.

The next two results are the two main tools that we will use to study the long time behavior of solutions and photon loss.
The first asserts that the solution operator to~\eqref{e:komp}--\eqref{e:noflux} is a contraction in $L^1$.
The second provides a comparison principle, without requiring a boundary condition at $x = 0$.
Both results are proved in Section~\ref{s:construct}, below.
\begin{lemma}[$L^1$ contraction]\label{l:l1contract}
  Let $n, m$ be two bounded, nonnegative solutions of~\eqref{e:komp}.
  Then for any $0 \leq s \leq t$, we have
  \begin{equation}\label{eqn:contract}
    \norm{n_t - m_t}_{L^1}
      + \int^t_s \abs[\big]{ n^2_\tau (0)-m^2_\tau(0)} \, d\tau
    \leq \norm{n_s - m_s}_{L^1}\,.
  \end{equation}
\end{lemma}
\begin{remark*}
  Here $n_\tau(0) = \lim_{x \to 0^+} n_\tau(x)$, which exists by Lemma~\ref{l:n0exist}.
\end{remark*}

\begin{lemma}[Weak comparison principle]\label{l:comparison}
  Let $T > 0$, $m \in L^\infty(Q_T)$ be a nonnegative sub-solution to~\eqref{e:komp}--\eqref{e:noflux}, and $n \in L^\infty(Q_T)$ be a nonnegative super-solution to~\eqref{e:komp}--\eqref{e:noflux}.
  If $m_0 \leq n_0$ then we must have $m_t \leq n_t$ for all $t \in [0, T]$.
\end{lemma}

We reiterate that our comparison principle \emph{does not} require the assumption $m \leq n$ at the boundary~$x = 0$.
Instead, it provides as a conclusion that $m_t(x) \leq n_t(x)$ for every $t > 0$ and $x \geq 0$, including at $x = 0$.
The notion of sub and super-solutions used in Lemma~\ref{l:comparison} are defined precisely in Definition~\ref{d:subsol}, below.

\subsection{Photon Loss, and Condensate Formation.}

We now turn to results concerning loss of photons, 
which corresponds to Bose--Einstein condensation at the level of approximation that the Kompaneets equation represents.
Throughout this section we will assume $n_0$ is a nonnegative bounded function satisfying~\eqref{e:x2n}, and $n$ is the global solution to~\eqref{e:komp}--\eqref{e:noflux} with initial data $n_0$.
Our first result is an explicit formula for the total photon number that was mentioned earlier.

\begin{proposition}[Loss formula]\label{p:lossFormula}
  Whenever $0 \leq s \leq t$ we have
  \begin{equation}\label{e:loss}
    N(n_t)  + \int_s^t n_\tau(0)^2 \, d\tau = N(n_s)\,.
  \end{equation}
\end{proposition}

As a result the total photon number can only decrease with time, and can only decrease through an outflux of zero-energy photons.
In fact, equation~\eqref{e:loss} shows that the change in the total photon number between time $0$ and $t$ is exactly $\int_0^t n_s(0)^2 \, ds$, and thus we may interpret $\int_0^t n_s(0)^2 \, ds$ as the mass of the Bose--Einstein condensate at time $t$.
Notice that since $n_s(0)^2$ is manifestly nonnegative, the total photon number can never increase through an influx at $x=0$.
That is, according to Kompaneets dynamics, photons may enter the Bose--Einstein condensate, but can not leave it.

In addition to the above physical interpretation, Proposition~\ref{p:lossFormula} is essential to obtaining the behavior of the flux at $x = 0$ (Lemma~\ref{l:intJ0}, above).
Thus we prove Proposition~\ref{p:lossFormula} in Section~\ref{s:construct}, before Lemma~\ref{l:intJ0}.
\medskip

Next we show that once a photon outflux at $x = 0$ starts, it will never stop.
Thus once the Bose--Einstein condensate forms, its mass will always strictly increase with time.
\begin{proposition}[Persistence]\label{p:persistence}
  There exists $t_*\in [0,\infty]$ such that $n_t(0)>0$ whenever $t> t_*$, and $n_t(0)=0$ whenever $0 < t<t_*$.
\end{proposition}

Due to Proposition~\ref{p:lossFormula}, $t_*$ is the time of onset of photon loss.
Of course there are situations (such as Corollary~\ref{c:mulim}, in the case where $n_0 \leq \hat n_0$) when $t_* = \infty$, and 
photon loss never occurs.
There are, however, a few scenarios under which one can prove $t_* < \infty$, so photon loss begins in finite time.
One scenario was constructed by Escobedo et\ al.\ in~\cite{EscobedoHerreroEA98}, where the solution develops a ``viscous shock'' (see Figure~\ref{f:formationEHV}).
Namely, for any $t_*, c_* > 0$, they produce solutions~$n$ such that
\begin{equation*}
  \lim_{t \to t_*^-} n_t(c (t_* - t)) = c_*
  \quad\text{if } c > c_*\,,
  \qquad\text{and}\qquad
  \lim_{t \to t_*^-} n_t(c (t_* - t)) = 0
  \quad\text{if } c < c_*\,.
\end{equation*}
For this solution $n_t(0)$ has a jump discontinuity at $t = t_*$, and $N(n_t)$ has a corner at $t = t_*$ (see Figure~\ref{f:formationEHV}).
Escobedo et al.\ produce such solutions with $N(n_t)$ arbitrarily small, showing this scenario is not related to an excess of 
photon number above the maximum equilibrium value $N(\hat n_0)$.
\begin{figure}[htb]
  \setkeys{Gin}{width=.3\linewidth}
  \includegraphics{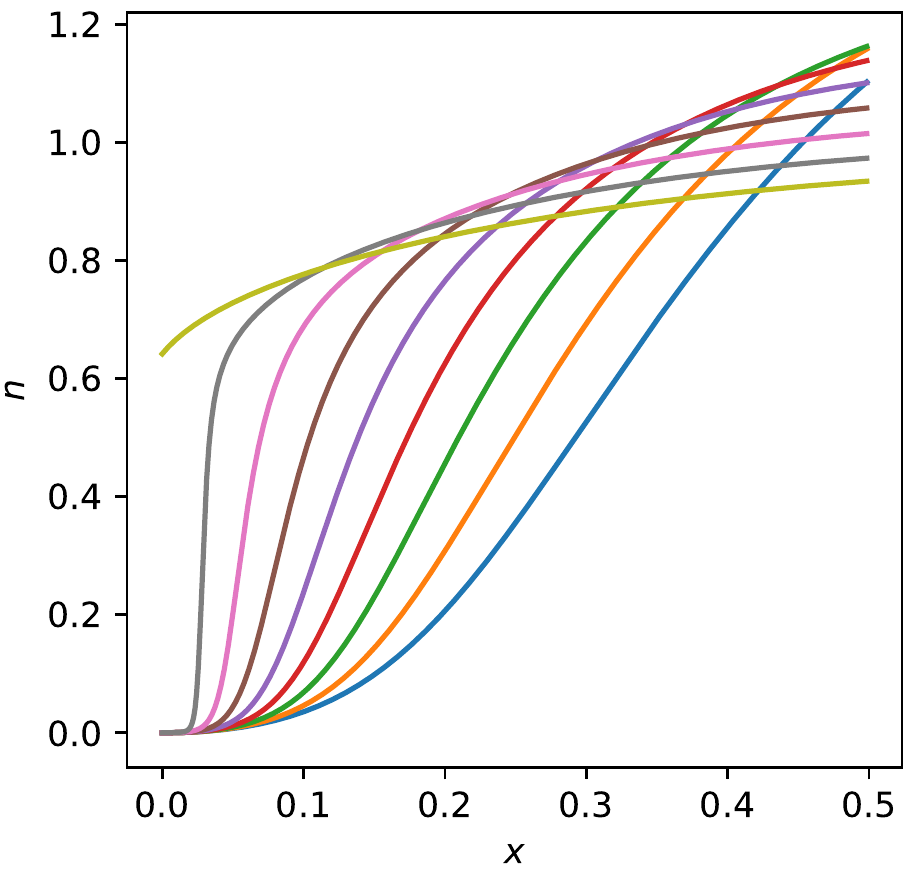}
  \quad
  \includegraphics{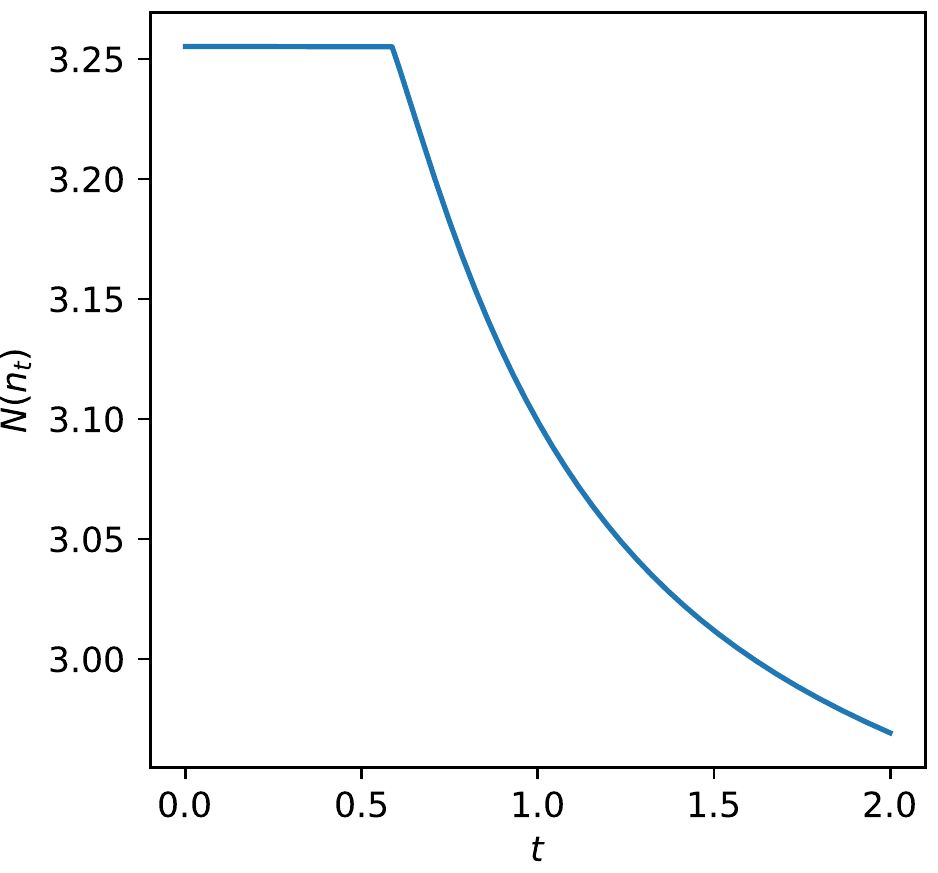}
  \quad
  \includegraphics{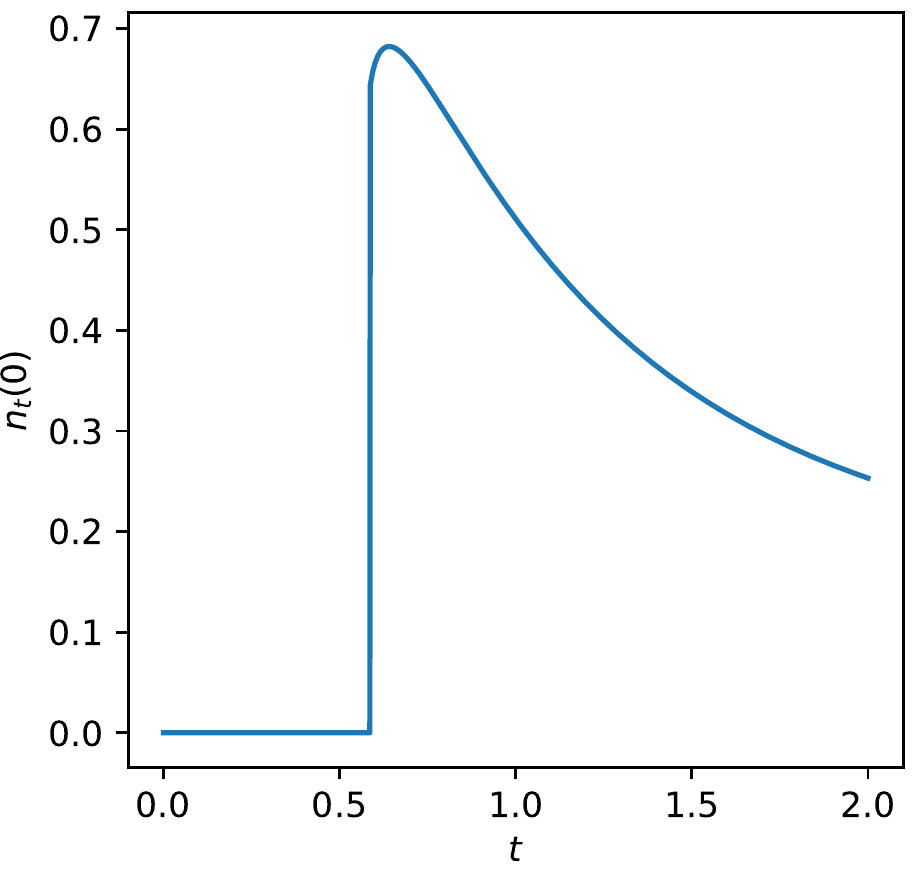}
  \caption{%
    \footnotesize
    Numerical simulations showing the onset of photon loss 
    through a ``viscous shock'' \`a la Escobedo et\ al.~\cite{EscobedoHerreroEA98}.
    Left: Profile of the solution $n_t(x)$ vs $x$ at various times close to $t_*$.
    Center: Total photon number $N(n_t)$ vs $t$ showing a corner at $t = t_*$.
    Right: Photon outflux at $x = 0$ vs $t$, showing a jump at $t = t_*$.
    (The numerical method used to generate these plots is described in Appendix~\ref{s:nmethod}, and the parameters used are listed in Remark~\ref{r:params}.)
  }
  \label{f:formationEHV}
\end{figure}

There are two other scenarios under which one can prove $t_* < \infty$.
The first is a mass condition that guarantees $t_* < \infty$ if the initial photon number larger than that the maximum photon number can be sustained in equilibrium.
While this is the natural physically expected behavior, it was not proved rigorously before.
We prove it here as a consequence of our long time convergence result (Theorem~\ref{t:lim}).

\begin{figure}[htb]
  \setkeys{Gin}{width=.3\linewidth}
  \includegraphics{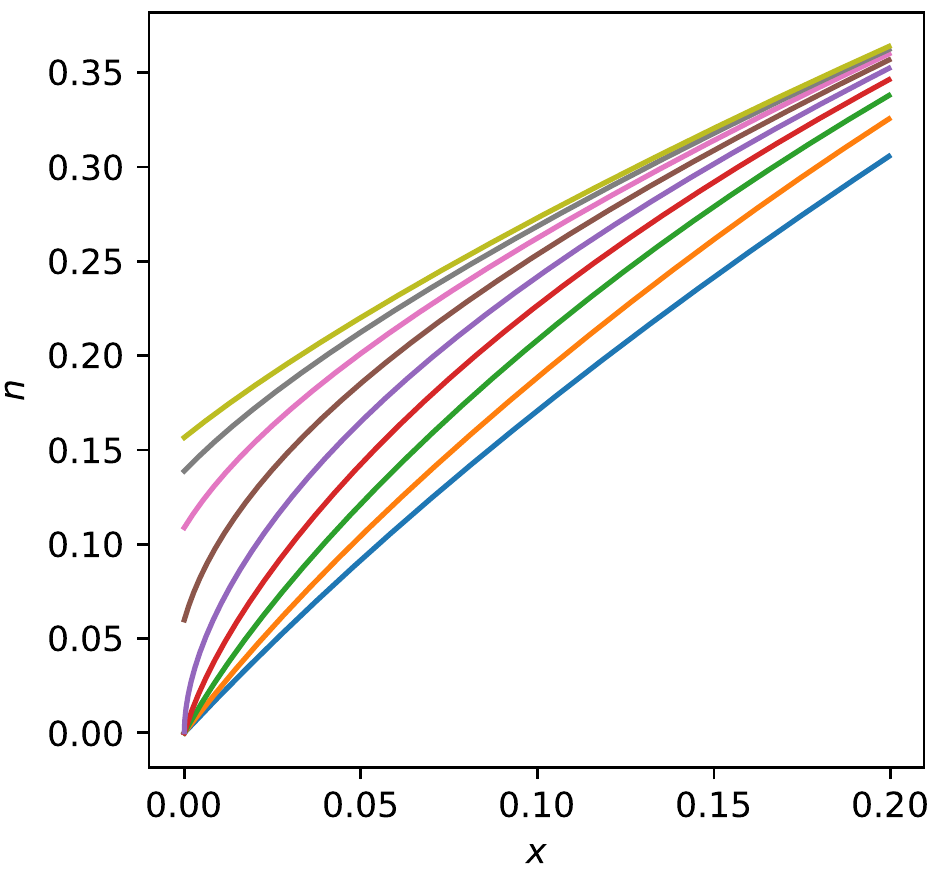}
  \quad
  \includegraphics{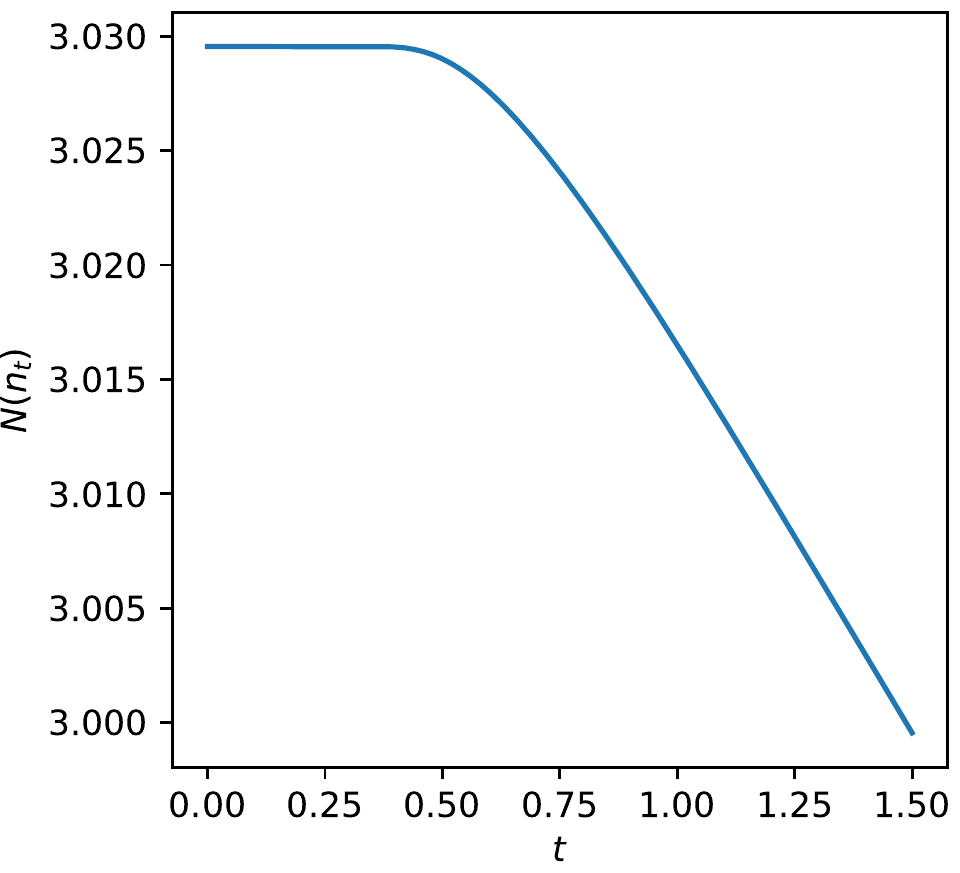}
  \quad
  \includegraphics{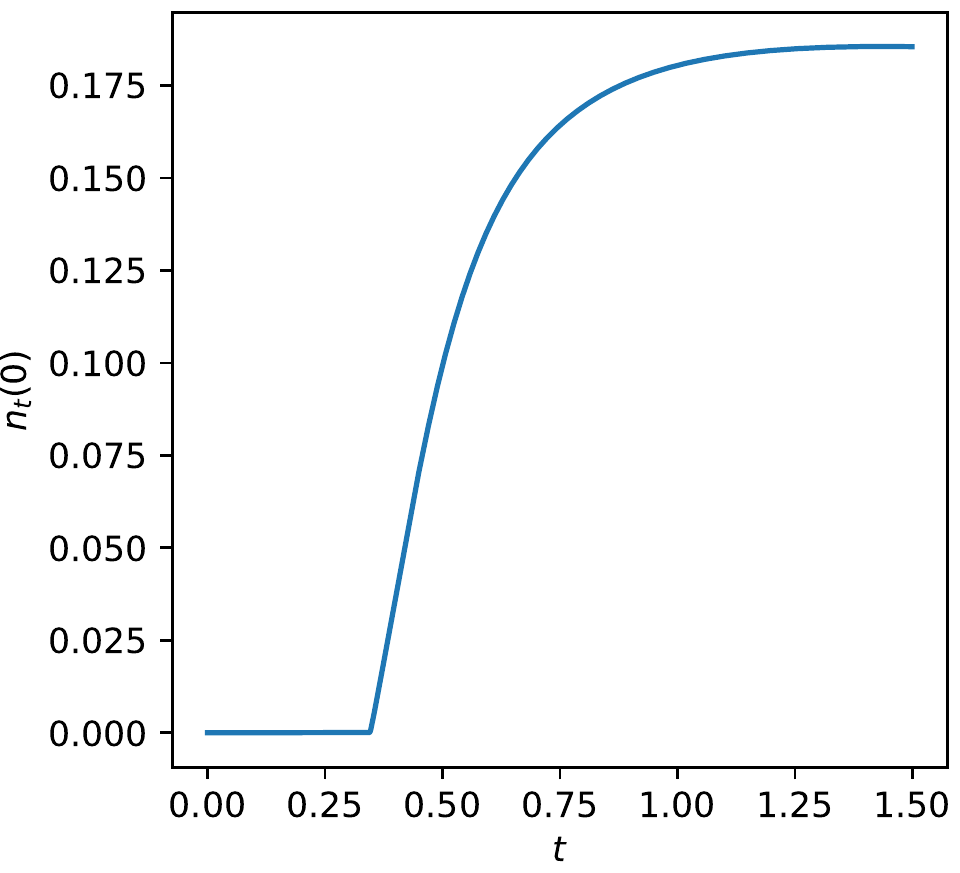}
  \caption{%
    \footnotesize
    Numerical simulations showing the onset of photon loss through the slope condition in Proposition~\ref{p:formation}.
    Left: Profile of the solution $n_t(x)$ vs $x$ at various times close to $t_*$.
    Center: Total photon number $N(n_t)$ vs $t$ showing a~$C^1$ transition at $t = t_*$.
    Right: Photon outflux at $x = 0$ vs $t$, showing a corner at $t = t_*$.
    (The numerical method used to generate these plots is described in Appendix~\ref{s:nmethod}, and the parameters used are listed in Remark~\ref{r:params}.)
  }
  \label{f:formationSlope}
\end{figure}
The second scenario we provide here is a slope condition that guarantees $t_* < \infty$, provided $\partial_x n_0 > 1$ (see Figure~\ref{f:formationSlope}).
To the best of our knowledge, this scenario hasn't been identified before.
In this case, in contrast to the viscous shock of Escobedo et\ al.~\cite{EscobedoHerreroEA98}, the photon outflux at $0$ can be continuous but not differentiable  in time at $t = t_*$.
Moreover, $\partial_x n_t(0) \to \infty$ as $t \to t_*^-$.

\begin{proposition}[Onset of loss]\label{p:formation}
  Let $t_*$ be the time given by Proposition~\ref{p:persistence}.
  \begin{enumerate}\reqnomode
    \item
      \emph{(Mass condition)} If $N(n_0) > N(\hat n_0)$ then $t_* < \infty$.
    \item
      \emph{(Slope condition)}
      If $\partial_x n_0(0)>1$, then $t_* \leq \bar t_*$, where
      \begin{equation}\label{e:tStarRiccati}
	\bar t_* \defeq \frac{1}{2} \ln\paren[\Big]{ \frac{\partial_x n_0(0)}{\partial_x n_0(0) -1} } \,. 
      \end{equation}
      Moreover, there exists $n_0$ for which $t_* = \bar t_*$ and the photon outflux $n_t(0)$ is continuous at $t = t_*$.
  \end{enumerate}
\end{proposition}

In the second scenario above, the profile of the initial photon number density away from the origin may initiate loss of photons 
before time $\bar t_*$.
It also allows the photon outflux to have a jump discontinuity at any time, before, at, or after $\bar t_*$, though.
Thus for arbitrary initial data with $\partial_x n_0(0) > 1$, we cannot expect photon loss to begin exactly at time 
$t_* = \bar t_*$, or the photon outflux to be continuous at time $t_*$.
We can, however, produce a family of initial data for which it is true that $t_* = \bar t_*$ precisely, and $n_t(0)$ is continuous at (but not necessarily after) time $t_*$, and this is the second assertion of Proposition~\ref{p:formation}.
The initial data we consider produces solutions where $n_{t_*}(\cdot)$ develops a square-root singularity at $x=0$, and $n_{\cdot}(0)$ develops a corner at $t = t_*$.
Generic initial data for which the photon outflux is continuous at time $t_*$ may exhibit different singular behavior at the point $(x, t) = (0, t_*)$, and we do not presently have a complete characterization.

Finally, we state one result that guarantees photon loss never occurs.
Namely, if the initial data lies entirely below the maximal stationary solution~$\hat n_0$, then 
total photon number is globally conserved and no condensate can form.
\begin{lemma}[Absence of loss]\label{l:absence}
  If $n_0 \leq \hat n_0$, then~$t_* = \infty$.
\end{lemma}

Lemma~\ref{l:absence} is already contained in work of Escobedo \& Mischler~\cite{EscobedoMischler01}.
We can, however, provide a short and direct proof of it here using the comparison principle. 
\begin{proof}[Proof of Lemma~\ref{l:absence}]
  Treating $\hat n_0$ as a super-solution and $n$ as a sub-solution, Lemma~\ref{l:comparison} implies that $n_t(x) \leq \hat n_0(x)$ for all $t > 0$, $x \geq 0$.
  Since $\hat n_0(0) = 0$, we must have $n_t(0) = 0$ for all $t > 0$, hence the total photon number $N[n_t]$ is constant.
\end{proof}

\subsection{Long Time Convergence.}\label{s:longtime}

We now study the behavior of solutions as~$t \to \infty$.
Our main result shows that as~$t \to \infty$, the solution must converge (in~$L^1$) to a stationary solution~$\hat n_\mu$.
The parameter~$\mu$ and the total loss in the photon number can be determined uniquely (but not explicitly) from the initial data.

\begin{theorem}[Long time convergence]\label{t:lim}
  Let $n_0$ be a nonnegative bounded function which is not identically $0$ and satisfies~\eqref{e:x2n}.
  If $n$ is the unique global solution of~\eqref{e:komp}--\eqref{e:noflux} with initial data $n_0$, then
  \begin{equation}\label{e:l1conv}
    \lim_{t \to \infty} \norm{ n_t - \hat n_\mu }_{L^1} = 0\,,
  \end{equation}
  where $\mu \in [0, \infty)$ is the unique number for which
  \begin{equation}\label{e:muDef}
    N(\hat n_\mu)=N(n_0)- \int_0^\infty n_t(0)^2 \, dt \,.
  \end{equation}
\end{theorem}

As mentioned above, while the parameter~$\mu$ can be determined uniquely from the identity~\eqref{e:muDef}, it can not in general be explicitly computed from the initial data $n_0$.
The best we can do presently is to obtain a non-trivial lower bound for $N(\hat n_\mu)$ for general initial data, and compute $\mu$ explicitly in two special cases.
We present this below.

\begin{corollary}\label{c:mulim}
  Let $n_0$ and $n$ be as in Theorem~\ref{t:lim}.
  \begin{enumerate}\reqnomode
    \item
      If $n_0 \geq \hat n_0$, then~\eqref{e:l1conv} holds with~$\mu = 0$, and the total loss in the photon number is precisely $N(n_0) - N(\hat n_0)$.

    \item
      If, on the other hand, $n_0 \leq \hat n_0$, then there is never any photon outflux at $x = 0$, and a Bose--Einstein condensate never forms.
      Consequently~\eqref{e:l1conv} holds for the unique $\mu \in [0, \infty)$ such that $N(n_0) = N( \hat n_\mu )$.
    \item
      In general, the total photon number in equilibrium is bounded below by
      \begin{equation}\label{e:NmuLower}
	N(\hat n_\mu) \geq \int_0^\infty (n_0 \varmin \hat n_0) \, dx > 0\,,
      \end{equation}
      where the notation $a \varmin b$ above denotes the minimum of $a$ and $b$.
  \end{enumerate}
\end{corollary}

The main tool used in the proof of Theorem~\ref{t:lim} is to use the fact that solutions dissipate a quantum entropy
\begin{equation}\label{e:HdefIntro}
  H(n) \defeq \int_0^\infty \paren[\Big]{
    xn + n\ln n -(n+x^2)\ln(n+x^2)+x^2 \ln (x^2)
  } \, dx\,.
\end{equation}
While this can be checked formally (see~\cite{CaflischLevermore86} or Section~\ref{s:largetime}), we are only able to prove~\eqref{e:HdefIntro} rigorously under a decay assumption on the initial data.
\begin{lemma}[Entropy dissipation]\label{ent}
  Suppose there exists $C_0 > 0$ such that
  \begin{equation}\label{e:expDecay}
    n_0(x) \leq C_0(1 + x^2) e^{-x}
    \quad\text{for all}\quad x > 0 \,.
  \end{equation}
  Then for any~$t > 0$ we have
  \begin{equation}\label{e:ent}
    \partial_t H(n_t)
    + \int_0^\infty \frac{J^2}{n(n+ x^2)} \, dx
    = 0.
  \end{equation}
\end{lemma}
\begin{remark}
  To prove existence of solutions to~\eqref{e:komp}--\eqref{e:noflux}, one only needs to assume the algebraic decay condition~\eqref{e:x2n}.
  However, to prove Lemma~\ref{ent}, we required the exponential decay assumption~\eqref{e:expDecay}.
  We are presently unaware if~\eqref{e:ent} holds for initial data that only satisfies~\eqref{e:x2n}, and our proof requires the strong exponential decay assumption~\eqref{e:expDecay}.
\end{remark}

Once entropy decay is established, we will prove Theorem~\ref{t:lim} using the $L^1$-contraction and LaSalle's invariance principle (see Section~\ref{s:largetime}, below).
We remark, however, our proof gives no information on the rate at which~$n_t \to \hat n_\mu$.
For exponentially decaying initial data, the entropy can be used to provide some information about the convergence rate, and we conclude by stating this result.

\begin{proposition}[Convergence rate]\label{p:rate}
  Suppose~$n_0$ satisfies~\eqref{e:expDecay}, and let~$\mu$ be as in Theorem~\ref{t:lim}.
  Then there exists a constant $C$ (depending only on~$C_0$) such that
  \begin{equation}\label{e:rate}
    \norm{ x^2 (n_t - \hat n_\mu) }^2_{L^1}
      \leq C \paren[\Big]{ H(n)-H(\hat n_\mu)+\mu \int_t^\infty n_s(0)^2 \, ds } \,
  \end{equation}
  for all~$t \geq 0$.
\end{proposition}

\subsection*{Plan of this paper.}
In Section~\ref{s:construct} we prove our results concerning properties of solutions
(Lemmas~\ref{l:n0exist}--\ref{l:comparison})
and the loss formula (Proposition~\ref{p:lossFormula}).
The proofs require an Oleinik type bound on the negative slope, and we also prove this in Section~\ref{s:construct}.
In Section~\ref{s:bec} we prove our results concerning onset, persistence and absence of photon loss (Proposition~\ref{p:persistence}, the second assertion in Proposition~\ref{p:formation} and Lemma~\ref{l:absence}).
In Section~\ref{s:largetime} we prove our results concerning the long time behavior (Theorem~\ref{t:lim}, Corollary~\ref{c:mulim}, Lemma~\ref{ent} and~\ref{p:rate}, and the first assertion in Proposition~\ref{p:formation}).
Finally we conclude this paper with two appendices.
Appendix~\ref{s:nmethod} describes the numerical method used to produce Figures~\ref{f:formationEHV}--\ref{f:formationSlope}.

\section{Construction and Properties of Solutions.}\label{s:construct}
\subsection{Negative Slope Bounds (Lemma~\ref{l:oleinik}).}
We begin by proving an Oleinik inequality, that guarantees a negative slope bound for solutions.
This will be used in the proofs of Lemmas~\ref{l:n0exist} and~\ref{l:intJ0}.
\begin{lemma}[Oleinik inequality]\label{l:oleinik}
Let $T > 0$, let $Q_T = (0, \infty) \times (0, T)$, and $n \in L^\infty(Q_T)$ be a nonnegative solution of~\eqref{e:komp}--\eqref{e:noflux} with initial data satisfying~\eqref{e:x2n}.
For every $(x, t) \in (0, \infty) \times (0, T]$, we have
\begin{equation}\label{e:dxnLower}
  \partial_x n \geq -\frac{1}{2t}-\frac{5x}{2} - \frac{\alpha}2 \,,
  \quad\text{for any }
  \alpha \geq {\sqrt{6 \norm{n}_{L^\infty(Q_T)} + 1} - 1 }\,.
\end{equation}
\end{lemma}
\begin{remark*}
  While Lemma~\ref{l:oleinik} provides a lower bound for $\partial_x n$, there is certainly no upper bound.
  In fact, we expect $\partial_x n$ to become singular at $x = 0$ at the time when loss of photons commences.
\end{remark*}
\begin{proof}[Proof of Lemma~\ref{l:oleinik}]
  Differentiating~\eqref{e:komp} with respect to $x$ and setting~$w = \partial_x n$ shows
  \begin{equation}\label{e:w}
    \Dt w =x ^2 \Dxx w +2\paren[\Big]{ n +x+\frac{x^2}{2} }\Dx w +2w(w+2x-1)+2n\,.
  \end{equation}
  The main idea behind the proof is to construct a suitable sub-solution to~\eqref{e:w}.
  For this, let $\delta > 0$ and let $Q'_\delta = (2\delta, 1/\delta) \times (2\delta, T)$.
  Define $\ubar{w}$ by $\ubar{w} \defeq -(\varphi + \psi)$, where
  \begin{equation}\label{e:z0psiDef}
    \varphi \defeq  \frac{\alpha}2 + \frac{5 x}{2} +\frac{1}{2(t - \delta) }\,,
    \qquad\text{and}\qquad
    \psi \defeq \frac{\delta e^{30t}}{(x - \delta)^{4}}\,.
  \end{equation}
  We claim~$\ubar{w}$ is the desired sub-solution to~\eqref{e:w} in~$Q'_\delta$.

  To see this, define the linear differential operator $\mathcal L_0$ by
  $$
    \mathcal L_0 v \defeq \Dt v -x^2 \partial_x^2 v - 2 \paren[\Big]{ n+x + \frac{x^2}{2} } \Dx v
    +(1-2x)2v \,,
  $$
  and observe that~\eqref{e:w} implies
  $
    \mathcal L_0 w = 2w^2 + 2n\,.
  $
  Thus for $u=w-\ubar{w}$, we have 
  \begin{equation}\label{e:ole1}
    \mathcal L_0 u
      =  2(u+\ubar w)^2 + 2n -\mathcal L_0 \ubar{w} 
      =  2u^2 + 4u\ubar{w} + 2\ubar{w}^2 + 2n +\mathcal L_0 (\varphi + \psi)  \,.
  \end{equation}
  We will now find a lower bound for the right hand side of this.
  
  First, we note that $\varphi\psi > 0$ and $2\partial_x\varphi=5$ in $Q'_\delta$, hence
  \begin{align}
  \nonumber
     & 2\ubar{w}^2+2n+ \mathcal L_0 \varphi  
     \\ \nonumber
     & >
     2(\psi^2+\varphi^2)-3n + 2\varphi(1-2x) 
     -5\paren[\Big]{x+\frac{x^2}{2}}
     -\frac1{2(t-\delta)^2}
     \\ \nonumber
     & = 2\psi^2-3n + \frac{ 1+ \alpha +3x }{t-\delta}
     + 3\alpha x + \frac12\alpha(\alpha+2)
    \\ \label{e:ole2}
      & \geq  2\psi^2  - 3 \norm{n}_{L^\infty(Q_T)} + \frac12\alpha(\alpha+2)
      \geq 2\psi^2\,.
  \end{align}
  The last inequality above followed by our choice of~$\alpha$.

Next, we note that since $e^{30t}>1$ and $n\geq 0$ we have
\begin{align}
\nonumber
2\psi^2+\mathcal L_0\psi 
&= \psi\brak[\Big]
 {2\psi + 30 - \frac{20x^2}{(x-\delta)^2}+(2n+2x+x^2)\frac4{x-\delta} + 2-4x}
 \\\nonumber 
&> \psi\brak[\Big]
 {\frac{2\delta}{(x-\delta)^4} + 30 -  \frac{20x^2}{(x-\delta)^2}+(2x+x^2)\frac4{x} + 2-4x}
 \\
  \label{e:ole3}
&= \psi\brak[\Big]
 {\,\frac{2}{\delta^3}
 \paren[\Big]{\frac{\delta}{x-\delta}}^4 
 + 40 - 20 
 \paren[\Big]{1+\frac{\delta}{x-\delta}}^2\, } >0\,,
\end{align}
provided $x > \delta$, and $\delta$ is sufficient small.
Using~\eqref{e:ole2} and~\eqref{e:ole3} in~\eqref{e:ole1} implies
\begin{equation}\label{e:L0U}
  \mathcal L_0 u \geq 2u^2  + 4u\ubar{w} \geq 4 \ubar{w} u\,.
\end{equation}

Finally, let
\begin{equation*}
  \partial Q'_\delta = \paren[\big]{ \set{2\delta, \tfrac{1}{\delta} } \times  [2\delta, T]  } \cup \paren[\big]{ \paren{2\delta, \tfrac{1}{\delta}} \times \set{2\delta} }
\end{equation*}
denote the parabolic boundary of~$Q'_\delta$.
Using~\eqref{e:x2dxnVanish} and~\eqref{e:z0psiDef} we see that $u > 0$ on $\partial Q'_\delta$ when~$\delta$ is sufficiently small.
Moreover, since $\ubar{w}$ and $2x - 1$ are bounded in $Q'_\delta$, the minimum principle (see for instance~\cite[\S7.1.4]{Evans98}) applies to the inequality~\eqref{e:L0U} and guarantees $u \geq 0$ in $Q'_\delta$.
Sending~$\delta \to 0$ concludes the proof.
\end{proof}

One immediate consequence of Lemma~\ref{l:oleinik} is that for any~$t > 0$ the limit $\lim_{x \to 0^+} n_t(x)$ must exist.
This is the content of Lemma~\ref{l:n0exist}, and we prove it here.

\begin{proof}[Proof of Lemma~\ref{l:n0exist}]
  Notice that for any $t>0$, \eqref{e:dxnLower} implies 
  \begin{equation*}
  \partial_x\paren[\Big]{n_t+ x\paren[\Big]{\frac\alpha2 + \frac1{2t}}+\frac{5x^2}4} \ge 0.
  \end{equation*}
  This immediately implies $\lim_{x \to 0} n_t(x)$ exists as claimed. 
\end{proof}
\subsection{Loss Formula and Flux Behavior at \texorpdfstring{$x = 0$}{x=0}.}
Next we prove Proposition~\ref{p:lossFormula} that provides an explicit expression relating the decrease in the total photon number to the outflux of zero-energy photons.
\begin{proof}[Proof of Proposition~\ref{p:lossFormula}]
  Integrating~\eqref{e:komp} on~$(x, R) \times (s, t)$ we see
  \begin{equation*}
    \int_x^R (n_t(y) - n_s(y)) \, dy
      = \int_s^t J(R, n_\tau) \, d\tau - \int_s^t (x^2 \partial_x n_\tau + (x^2 - 2x) n_\tau + n_\tau^2) \, d\tau \,.
  \end{equation*}
  The first term on the right vanishes as $R \to \infty$ (equation~\eqref{jr}).
  The second term requires care: we can't directly send $x \to 0$ as we don't know the behavior of~$x^2 \partial_x n_\tau(x)$ at this stage.
  We instead average this term for~$x \in (\epsilon, 2\epsilon)$, and then send $\epsilon \to 0$.
  Using the notation~$\avint_\epsilon^{2\epsilon} f \, dx$ to denote the averaged integral~$\frac{1}{\epsilon} \int_\epsilon^{2\epsilon} f \, dx$, we observe
  \begin{equation}\label{e:avN}
    \avint_\epsilon^{2\epsilon}
      \int_x^\infty (n_t(y) - n_s(y)) \, dy \, dx
      = - \int_s^t \avint_{\epsilon}^{2\epsilon} (x^2 \partial_x n_\tau + (x^2 - 2x) n_\tau + n_\tau^2) \, dx \, d\tau \,.
  \end{equation}
  Notice that for $\tau>0$ fixed,
  \begin{equation*}
    \avint_\epsilon^{2\epsilon} x^2 \partial_x n_\tau
    = \frac{1}{\epsilon} \brak[\Big]{ x^2 n_\tau }_\epsilon^{2\epsilon} 
      - 2 \avint_\epsilon^{2\epsilon} x n_\tau \, dx\,,
  \end{equation*}
  which is uniformly bounded and vanishes as $\epsilon \to 0$.
  The dominated convergence theorem now implies
  \begin{equation*}
    \lim_{\epsilon \to 0} \int_s^t \avint_\epsilon^{2\epsilon} x^2 \partial_x n_\tau \, d\tau = 0\,.
  \end{equation*}
  Of course, Lemma~\ref{l:n0exist} and the dominated convergence theorem also implies
  \begin{equation*}
    \lim_{\epsilon \to 0} \int_s^t \avint_{\epsilon}^{2\epsilon} ((x^2 - 2x) n_\tau + n_\tau^2) \, dx \, d\tau
    = \int_s^t n_\tau(0)^2 \, d\tau\,.
  \end{equation*}
  By the fundamental theorem of calculus the left hand side of~\eqref{e:avN} converges to $N(n_t) - N(n_s)$ as $\epsilon \to 0$.
  Thus sending~$\epsilon \to 0$ in~\eqref{e:avN} yields~\eqref{e:loss} as claimed.
\end{proof}

Using the loss formula (Proposition~\ref{p:lossFormula}) in conjunction with the Oleinik inequality (Lemma~\ref{l:oleinik}) we can now prove Lemma~\ref{l:intJ0} concerning convergence of the flux $J(x, n_\tau)$ as $x \to 0$.
\begin{proof}[Proof of Lemma~\ref{l:intJ0}]
  Observe first
  \begin{equation*}
    \lim_{x \to 0} \int_s^t J(x, n_\tau) \, d\tau
      = \lim_{x \to 0} \int_x^\infty \paren[\big]{ n_s(y) - n_t(y) }
      = N(n_s) - N(n_t)
      = \int_s^t n_\tau(0)^2 \, d\tau \,,
  \end{equation*}
  proving~\eqref{e:intJ0}.
  Here the first equality follows from~\eqref{e:komp} and~\eqref{jr}, and the last equality follows from Proposition~\ref{p:lossFormula}.

  In order to prove the stronger convergence stated in~\eqref{e:a0}, we require $s > 0$.
  Let
  \begin{equation}\label{e:varphiDef}
    \varphi_\tau(y) \defeq  \frac{\alpha}2 + \frac{5 y}{2} +\frac{1}{2\tau }\,,
  \end{equation}
  where~$\alpha$ is as in~\eqref{e:dxnLower} and note that Lemma~\ref{l:oleinik} implies $\partial_x n_t \geq -\varphi_t $.
  Thus
  \begin{equation*}
    \abs{\partial_x n_t} \leq \partial_x n_t + 2 {\varphi_t}\,,
  \end{equation*}
  and hence
  \begin{align*}
    \int_s^t x^2 \abs{\partial_x n_\tau } \, d\tau
      &\leq \int_s^t x^2 \partial_x n_\tau  \, d\tau
	+ 2 \int_s^t x^2 {\varphi_\tau } \, d\tau
    \\
      &= \int_s^t (J(x, n_\tau) - n_\tau^2(x) ) \, d\tau
	+ \int_s^t \paren[\big]{
	      2 x^2 {\varphi_\tau}
	      -  (x^2 - 2x) n_\tau
	    }
	    \, d\tau\,.
  \end{align*}
  Using \eqref{e:intJ0} and the dominated convergence theorem the right hand side vanishes as $x \to 0$.
  This proves the first identity in~\eqref{e:a0}.
  The second identity follows immediately from the first and the dominated convergence theorem.
\end{proof}

\subsection{Contraction (Lemma~\ref{l:l1contract}).}
Our aim in this section is to prove that the solution operator to~\eqref{e:komp}--\eqref{e:noflux} is an~$L^1$ contraction (Lemma~\ref{l:l1contract}).
\begin{lemma}\label{l:contraction}
  Let $T > 0$, $m, n \in L^\infty(Q_T)$ be two nonnegative  solutions to \eqref{e:komp}--\eqref{e:noflux} whose initial data satisfies~\eqref{e:x2n}.
  For any $0 < r < R < \infty$, and any $0<s<t \leq T$, we have
  \begin{align}\label{eqn:contract1}
    \nonumber
    \MoveEqLeft
    \int^R_r \abs{ n_t(x)-m_t(x)} \, dx
    + \int^t_s\sgn\paren[\big]{
	n_\tau(r)-m_\tau(r)}
      \paren[\big]{
	J(r,n_\tau)-J(r,m_\tau)
      }
      \, d\tau
    \\
    \nonumber
    &\leq \int^R_r \abs{ n_s(x)-m_s(x) } \, dx
    \\
    &\phantom{\leq}
      +\int^t_s \sgn\paren[\big]{ n_\tau(R)-m_\tau(R) }
	\paren[\big]{ J(R,n_\tau)-J(R,m_\tau) }
      \, d\tau \, .
  \end{align}
\end{lemma}

Momentarily postponing the proof of Lemma~\ref{l:contraction}, we prove Lemma~\ref{l:l1contract}.
\begin{proof}[Proof of Lemma~\ref{l:l1contract}]
  In order to prove Lemma~\ref{l:l1contract}, we only need to send $r \to 0$ and $R \to \infty$ in~\eqref{eqn:contract1}.
  Using~\eqref{e:x2nPtwise}, \eqref{e:x2dxnVanish} the last term on the right of~\eqref{eqn:contract1} vanishes as $R \to \infty$.

  The convergence as $r \to 0$ requires care: While $J(r, n_\tau) \to n_\tau(0)^2$ in $L^1((s, t))$, the sign function is discontinuous and the pre-factor $\sign(n_\tau(r) - m_\tau(r))$ may not converge.
  Expanding out the flux explicitly, however, allows us to still compute the limit as $r \to 0$.
  Indeed let~$J_{\lin}$ be the linear terms in the flux~$J$.
  That is, define
  \begin{equation}\label{e:Jlin}
    J_\lin(x, n) \defeq x^2 \partial_x n + (x^2 - 2x) n\,,
  \end{equation}
  and note $J(x, n) = J_\lin(x, n) + n^2$.
  Now
  \begin{align*}
    \MoveEqLeft
      \int_s^t
	\sign( n_\tau(r) - m_\tau(r) ) (J(r, n_\tau) - J(r, m_\tau))
	\, d\tau
      \\
      &= \int_s^t
	  \abs{n_\tau(r) - m_\tau(r)}^2
	    \, d\tau
      \\
	&\qquad+ \int_s^t
	  \sign( n_\tau(r) - m_\tau(r) ) (J_{\lin}(r, n_\tau) - J_{\lin}(r, m_\tau))
	  \, d\tau.
  \end{align*}
  The first term on the right converges to $\int_s^t (n_\tau(0)^2 - m_\tau(0)^2) \, d\tau$ as $r \to 0$.
  For any $0 < s \leq t$, \eqref{e:a0} and the dominated convergence theorem imply that the second term vanishes as $r \to 0$.
  This proves~\eqref{eqn:contract} for any $0 < s \leq t$.
  Using continuity in $L^1$ and sending $s \to 0$ concludes the proof.
\end{proof}

It remains to prove Lemma~\ref{l:contraction}.
Before going through the proof, we first perform a formal calculation showing why~\eqref{eqn:contract1} is expected.
Let $w=n-m$ and note
\begin{equation}\label{e:w1}
  \partial_t w -\partial_x (J(x, n)-J(x, m))=0 \,.
\end{equation}
Multiplying by~$\sign(w)$ and integrating in space gives
\begin{align}\label{e:wFormal}
  \nonumber
  \MoveEqLeft
  \int_r^R \partial_t \abs{w} \, dx
    - \brak[\Big]{ \sign(w) ( J(x, n)  - J(x, m) ) }_r^R
  \\
    &\qquad= -\int_r^R \partial_x \sign(w) ( J(x, n)  - J(x, m) ) \, dx \,.
\end{align}
We will show (using the structure of~$J$) that the right hand side is negative.
Once this is established, integrating in time will yield~\eqref{eqn:contract1} as desired.

To make this argument rigorous, we will regularize $\sign(w)$, and explicitly check that the right hand side of~\eqref{e:wFormal} is indeed negative.
As we will shortly see we only obtain a one sided bound for this term after regularization: while it is certainly negative, it need not vanish.

\begin{proof}[Proof of Lemma~\ref{l:contraction}]
  Let $\sgn_\eps$ be an odd smooth increasing function on $\mathbb{R}$ such that
  \begin{equation}\label{eqn:sgnepsdef}
    \sgn_\eps(x)\defeq
      \begin{cases}
	1 & x>\eps \,,\\
	-1 & x < -\eps\,,
  \end{cases}
  \end{equation}
  given by $\sgn_\eps(x)=2\int_0^x\eta_\eps(z)\,dz$ where $\eta_\eps$ is a standard mollifier~\cite{Evans98},
  and define
  \begin{equation*}
    \zeta_{\eps}(x) = \int_0^x \sgn_\epsilon(y) \, dy \,.
  \end{equation*}
  Note that~$\zeta_\epsilon$ is a smooth convex even function with~$\zeta_\epsilon(0) = 0$.

  Multiplying equation~\eqref{e:w1} by~$\sign_\epsilon(w)$ and integrating by parts yields
  \begin{align*}
    \MoveEqLeft
    \int_s^t \int_r^R \partial_t \zeta_\eps(w) \, dx \, d\tau - \int_s^t \sgn_\eps(w) (J(x, n)-J(x, m)\Bigr|_r^R d\tau
    \\
    & =- \int_s^t \int_r^R \partial_x \sgn_\eps(w)
      \paren[\big]{ J(x, n)-J(x, m) } \,  dx \, d\tau \,.
  \end{align*}
  We note that $\zeta_\epsilon(w)$ increases to $|w|$, and $\sgn_\eps(w) \to \sgn(w)$.
  Thus to complete the proof, it suffices to find an upper bound for the right hand side that vanishes as $\epsilon \to 0$.
  Using Young's inequality, we have  
  \begin{align*}
  \MoveEqLeft
  -\int_s^t \int_r^R  \sgn_\eps'(w)\partial_x w (x^2 \partial_x w +(x^2-2x)w +w(m+n)) \, dx \, d\tau \\
  & \leq - \int_s^t \int_r^R  \sgn_\eps'(w) \left(\frac{x^2}{2} (\partial_x w)^2 -\frac{w^2}{2x^2} (x^2-2x  +m+n)^2 \right) \, dx \, d\tau\\
  &  \leq  \frac{1}{2r^2} \int_s^t \int_r^R \sgn_\eps'(w) w^2 (R^2+ \norm{m}_{L^\infty(Q_T)} + \norm{n}_{L^\infty(Q_T)} )^2 \, dx \, d\tau. 
\end{align*}
Note that for any $z \in \R$, we must have $z^2 \sign_\epsilon'(z) \leq \epsilon$.
Indeed, if $\abs{z} \geq \epsilon$, then $\sign_\epsilon'(z) = 0$.
On the other hand, if $\abs{z} < \epsilon$, then $z^2 \sign_\epsilon'(z) =2\epsilon^2 \eta_\eps(z) \leq C\epsilon$.
Using this in the above yields
\begin{multline*}
  -\int_s^t \int_r^R  \sgn_\eps'(w)\partial_x w (x^2 \partial_x w +(x^2-2x)w +w(m+n)) \, dx \, d\tau \\
  \leq C\eps\frac{(R^2+ \norm{m}_{L^\infty(Q_T)} + \norm{n}_{L^\infty(Q_T)})^2}{2r^2}(t-s)(R-r) \,,
  \end{multline*}
  which vanishes as~$\epsilon \to 0$.
  This completes the proof.
\end{proof}

\subsection{Comparison (Lemma~\ref{l:comparison}).} 
This section is devoted to the proof of the comparison principle.
We begin by stating the definitions of the sub and super solutions we use.
\begin{definition}\label{d:subsol}
  Let $Q = (0, \infty) \times (0, \infty)$, and $n \in C^{2,1}(Q)$ be a function such that $n_t \to n_0$ in $L^1([0, \infty))$, and whenever $0 < s \leq t$ we have
  \begin{equation}\label{e:vanish0}
    \lim_{x \to 0^+} n_t(x) = n_t(0)\,,
    \qquad
    \lim_{x \to 0^+} \int_s^t \abs{x^2 \partial_x n_\tau} \, d\tau  = 0\,.
  \end{equation}
  We say $n$ is a \emph{sub-solution} to~\eqref{e:komp}--\eqref{e:noflux} if whenever $0 < s \leq t$ we have
  \begin{gather*}
    \partial_t n  \leq \partial_x J(x, n)\,,
    \qquad\text{and}\qquad
    \limsup_{x \to \infty} \int_s^t J(x, n_\tau)\, d\tau \leq 0\,.
  \end{gather*}
  We say $n$ is a~\emph{super-solution} to~\eqref{e:komp}--\eqref{e:noflux} if whenever $0 < s \leq t$ we have
  \begin{equation*}
    \partial_t n  \geq \partial_x J(x, n)\,,
    \qquad\text{and}\qquad
    \liminf_{x \to \infty} \int_s^t J(x, n_\tau)\, d\tau \geq 0\,.
  \end{equation*}
\end{definition}
\begin{remark}
  We note that the (globally unique) solutions provided by Theorem~\ref{t:exist} are both sub and super-solutions in the sense of Definition~\ref{d:subsol}.
  Certainly by~\eqref{e:komp} we have $\partial_t n = \partial_x J$.
  For the flux at infinity note that the vanishing conditions~\eqref{e:x2nPtwise} and \eqref{e:x2dxnVanish} imply
  \begin{equation*}
    \lim_{x \to \infty} \int_s^t J(x, n_\tau) \, d\tau = 0\,,
  \end{equation*}
  as needed.
\end{remark}
\begin{remark}
  As with the classical theory, we can relax the requirement that $n \in C^{2,1}(Q)$.
  For our purposes it will suffice to consider functions that are $C^{2,1}$, except on finitely many, non-degenerate, disjoint $C^2$ curves of the form $x = \gamma(t)$.
  At each point $(x, t)$ on one of these curves  we require continuity.
  Moreover, we require sub-solutions to satisfy the ``V corner'' condition
  \begin{subequations}
  \begin{equation}\label{e:cornerSubSol}
    -\infty < \partial_{x}^- n_{t}(x) \leq \partial_{x}^+ n_{t}(x) < \infty\,,
  \end{equation}
  and super-solutions to satisfy the ``inverted V corner'' condition
  \begin{equation}\label{e:cornerSupSol}
    \infty > \partial_{x}^- n_{t}(x) \geq \partial_{x}^+ n_{t}(x) > -\infty\,.
  \end{equation}
  \end{subequations}

  While a more general notion based on viscosity solutions is possible, it is unnecessary for our purposes since all the sub and super-solutions we construct are in the above form.
\end{remark}
\medskip

We now provide some intuition as to why the comparison principle (Lemma~\ref{l:comparison}) holds.
Suppose momentarily
\begin{equation}\label{e:nmStrict}
  \partial_t m \leq \partial_x J(x, m)\,,
  \qquad
  \partial_t n > \partial_x J(x, n)\,,
\end{equation}
and that $x \partial_x m$, $x^2 \partial_x^2 m$, $x \partial_x n$ and $x^2 \partial_x^2 n$ all vanish as $x \to 0$.
Let $t_0$ be the first time at which $n_{t_0}(0) = m_{t_0}(0)$.
A standard comparison principle argument (see for instance~\cite[Th. 2.6.16]{Friedman64}) shows $n \geq m$ for all $(x, t) \in [0, t_0] \times [0, \infty)$, and hence $\partial_t m_{t_0}(0) \geq \partial_t n_{t_0}(0)$ and $\partial_x (m_{t_0}(0)^2) \leq \partial_x (n_{t_0}(0)^2)$.
Using~\eqref{e:nmStrict} and our vanishing assumptions this would imply $\partial_t m_{t_0}(0) < \partial_t n_{t_0}(0)$, which is a contradiction.

In order to make the above argument rigorous, we would have to show that the terms $x \partial_x n$ and $x^2 \partial_x n$ vanish as $x \to 0$.
Presently we don't know whether or not either of these conditions holds, and the most we can prove (equation~\eqref{e:a0}) is not strong enough to make the above argument work.
To circumvent this issue, we instead prove Lemma~\ref{l:comparison} using the technique used to prove the weak maximum principle.

\begin{proof}[Proof of Lemma~\ref{l:comparison}]
  We begin by providing a formal argument.
  Let $w = n - m$ and observe
  \begin{equation*}
    \partial_t w - \partial_x (J(x, n) - J(x, m)) \geq 0\,.
  \end{equation*}
  Multiplying this by ~$-\one_{\set{w\leq 0}}$ and integrating yields, 
  since $w^-= -\one_{\set{w\leq 0}} w$, 
  \begin{align*}
    \MoveEqLeft
    \int_0^\infty \partial_t w^- \, dx
      + \brak[\Big]{ \one_{\set{w \leq 0}} ( J(x, n)  - J(x, m) ) }_0^\infty
    \\
      &\qquad\leq \int_0^\infty \partial_x \one_{\set{w \leq 0}} ( J(x, n)  - J(x, m) ) \, dx \,.
  \end{align*}
  Since $n$ is a super-solution and~$m$ is a sub-solution, we know
  \begin{equation*}
    \liminf_{x \to \infty} J(x, n) - J(x, m) \geq 0 \,.
  \end{equation*}
  Consequently,
  \begin{align*}
    \MoveEqLeft
    \int_0^\infty \partial_t w^- \, dx
      - \one_{\set{w \leq 0}} 
      (n(0)^2 - m(0)^2)
      \leq \int_0^\infty \partial_x \one_{\set{w \leq 0}} ( J(x, n)  - J(x, m) ) \, dx \,.
  \end{align*}
  As before, we claim the right hand side is negative.
  Once this is established, integrating in time immediately yields
  \begin{equation}\label{e:wMinus}
    \int_0^\infty w_t(x)^- \, dx
    + \int^t_s w_\tau(0)^- (n_\tau(0) + m_\tau(0) ) \, d\tau
    \leq \int_0^\infty w_s(x)^- \, dx\,.
  \end{equation}
  Sending $s \to 0$ we see that $w_t^- = 0$ for all $t > 0$, forcing $m_t(x) \leq n_t(x)$ for all $x, t \geq 0$.

  To make the above formal argument rigorous, we use the same regularization procedure as Lemma~\ref{l:contraction}.
  Let $0 < r < R < \infty$, and $0 < s < t \leq T$, and let $\mathcal H_\epsilon$ be a smooth increasing function such that
  \begin{equation*}
    \mathcal H_\epsilon(x) \defeq
      \begin{cases}
	-1 & x < -\epsilon  \,,
	\\
	0 & x > 0 \,,
      \end{cases}
  \end{equation*}
  given by $\mathcal H_\eps(x) = \int_{-\eps}^{\eps+2x}\eta_\eps(z)\,dz$ from the standard mollifier $\eta_\eps$,
  and let
  \begin{equation*}
    \zeta_\epsilon(x) = \int_0^x \mathcal H_\epsilon(y) \, dy\,.
  \end{equation*}
  By following the proof of Lemma~\ref{l:contraction}, we deduce an analog
  to \eqref{eqn:contract1}, namely
  \begin{align}\label{eqn:compare1}
    \nonumber
    \MoveEqLeft
    \int^R_r w_t(x)^- \, dx
    - \int^t_s
    \one_{\set{w \leq 0}} 
      \paren[\big]{
	J(r,n_\tau)-J(r,m_\tau)
      }
      \, d\tau
    \\
    &\leq \int^R_r w_s(x)^- \, dx
      -\int^t_s \one_{\set{w \leq 0}} \paren[\big]{ J(R,n_\tau)-J(R,m_\tau) }
      \, d\tau \, .
  \end{align}
Now sending~$r \to 0$ (using~\eqref{e:vanish0}) and~$R \to \infty$ as in the proof of Lemma~\ref{l:l1contract}, we obtain~\eqref{e:wMinus} as desired.
  This finishes the proof.
\end{proof}

\subsection{Pointwise Bounds.}\label{s:pointwise}

The comparison principle (Lemma~\ref{l:comparison}) allows us to easily obtain pointwise upper bounds, provided we can find suitable super-solutions.
The key to many of our estimates is a family of stationary super-solutions that allows us time independent bounds on solutions with exponentially decaying initial data.
We present this next.
\begin{figure}[htb]
  \includegraphics[width=.75\linewidth]{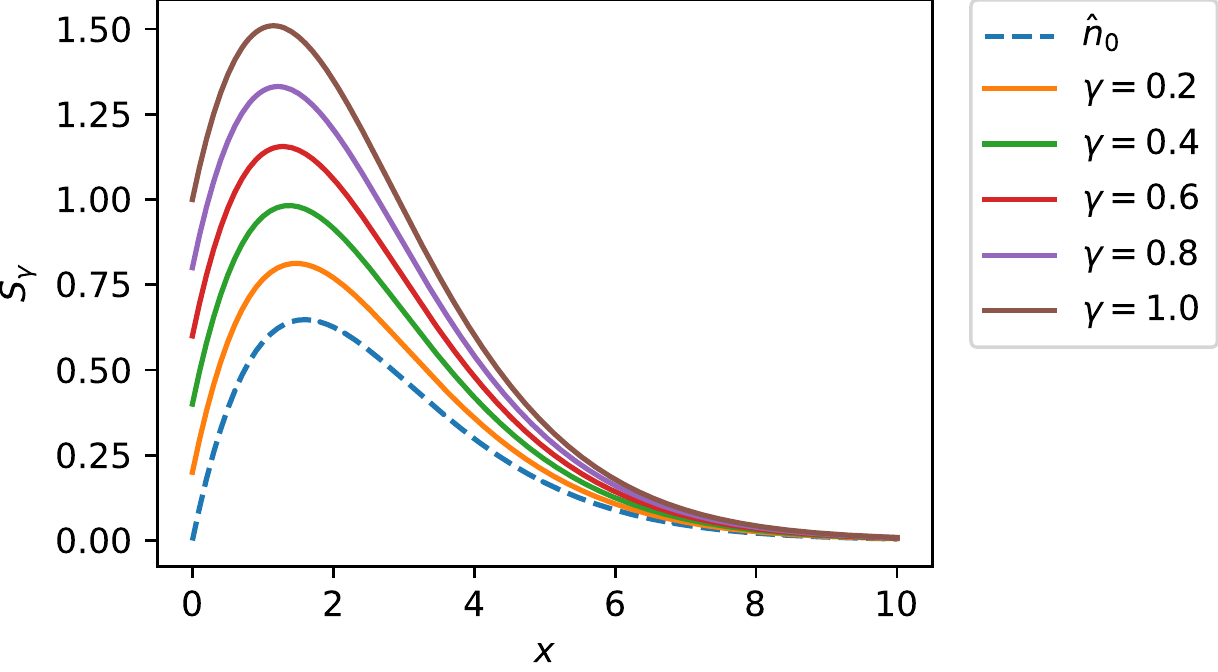}
  \caption{Plots of the stationary super-solution~$S_\gamma$ for various values of~$\gamma$.}
\end{figure}

\begin{lemma}\label{l:uss}
  For any~$\gamma \geq 0$ the function $S_\gamma$ defined by
  \begin{equation}\label{d:S}
    S_\gamma(x) =\hat n_0(x)  + \gamma m(x)\,,
    \quad\text{where}\quad
    m(x)
      =\frac{x^2e^x}{(e^x-1)^2}\,,
  \end{equation}
  is a stationary super-solution to~\eqref{e:komp}--\eqref{e:noflux}.
\end{lemma}
An immediate consequence of Lemmas~\ref{l:comparison} and~\ref{l:uss} is a uniform in time upper bound for any solution of~\eqref{e:komp}--\eqref{e:noflux} with sufficiently rapidly decaying initial data.
\begin{corollary}\label{c:Sbound}
  Let $n$ be the solution to~\eqref{e:komp}--\eqref{e:noflux} with initial data~$n_0 \geq 0$.
  If for some~$\gamma \geq 0$ we have~$n_0 \leq S_\gamma$, then we must have $n_t(x) \leq S_\gamma(x)$ for all $t \geq 0$, $x \geq 0$.
\end{corollary}
\begin{remark}
  Corollary~\ref{c:Sbound} applies to any initial data that can be bounded by $C(1+ x^2) e^{-x}$ for some $C \geq 0$, as in~\eqref{e:expDecay}.
\end{remark}

\begin{proof}[Proof of Lemma~\ref{l:uss}]
  Given the formula~\eqref{d:S}, one can directly differentiate and check that~$S_\gamma$ is a super-solution.
  A more illuminating proof is as follows:
  Observe
  \begin{equation*}
    J(x, S_\gamma) = J(x, \hat n_0) + J_\lin(x, \gamma m) + \gamma^2 m^2 + 2 \gamma m \hat n_0\,,
  \end{equation*}
  where $J_{\lin}$ (defined in equation~\eqref{e:Jlin}) is the linear terms in~$J$.
  Since $J(x, \hat n_0) = 0$, we will now look for functions~$m$ for which
  \begin{equation*}
    J(x, S_\gamma) = \gamma^2 m^2 \,.
  \end{equation*}
  For such functions we obtain the linear ODE
  \begin{equation*}
    J_\lin(x, m) + 2 m \hat n_0 = 0\,,
  \end{equation*}
  which simplifies to
  \begin{equation}\label{e:mEq}
    x^2 \partial_x m + (x^2 - 2x + 2\hat n_0)m = 0\,,
  \end{equation}
  Solving this ODE with the normalization~$m(0) = 1$ yields the formula for~$m$ in~\eqref{d:S}.

  Now to check~$S_\gamma$ is a stationary super-solution, we need to verify
  \begin{equation*}
    \partial_x J(x, S_\gamma) \leq 0
    \qquad\text{and}\qquad
    \lim_{x \to \infty} J(x, S_\gamma) \geq 0\,.
  \end{equation*}
  The second condition is true because $J(x, S_\gamma) = \gamma^2 m^2 > 0$.
  To check the first condition we note from~\eqref{d:S} 
  that~$m>0$, and from~\eqref{e:mEq} it follows
  \begin{equation*}
    \frac{\partial_x m}{m}
      = \frac{2x - x^2 - 2\hat n_0}{x^2}
      = \frac{g(x)}{x(e^x - 1)}\,,
    \quad\text{where}\quad
      g(x) \defeq (2 -x)(e^x - 1) - 2x \,.
  \end{equation*}
  Notice $g(0) = 0$, $g'(0) = 0$, and $g''(x) = - x e^x \leq 0$, which forces $g(x) \leq 0$ for all $x \geq 0$.
  Consequently, $\partial_x m \leq 0$ and hence $\partial_x J(x, S_\gamma) = 2 \gamma^2 m \partial_x m \leq 0$, concluding the proof.
\end{proof}

In case the initial data does not decay exponentially, we can still obtain explicit time-independent 
pointwise bounds.  The super-solutions, however, are not as natural as the $S_\gamma$ defined in~\eqref{d:S}.
\begin{proposition}\label{p:nBdd}
  For the solution~$n$ given by Theorem~\ref{t:exist}, $(1+x^2)n(x,t)$ is uniformly bounded on 
  $Q = (0, \infty)\times(0, \infty)$.
  Moreover, the uniform decay properties \eqref{e:x2nPtwise} and \eqref{e:x2dxnVanish} stated in Theorem~\ref{t:exist} hold with $T=\infty$.
\end{proposition}
\begin{remark}
  Theorem~2 in~\cite{EscobedoHerreroEA98} already asserts that the solution is uniformly bounded in time.
  However, we were unable to verify the proof given. It asserts that a function of the 
  form $\beta(1\varmin x^{-2})$ is a super-solution to equation (3.2) of \cite{EscobedoHerreroEA98}, 
  which is a truncated form of equation \eqref{e:komp} above.  
  This does not work if $\beta$ is a fixed constant independent of both time 
  and truncation.
Instead, the function $\beta e^{4t}(1\varmin x^{-2})$ (corresponding to (2.3) in \cite{EscobedoHerreroEA98}) 
works as a uniform super-solution and provides local in time bounds as we stated in Theorem~\ref{t:exist}.
\end{remark}
\begin{proof}[Proof of Proposition~\ref{p:nBdd}]
  Define
  \begin{equation*}
    Z(x) =
      \begin{dcases}
	e^{20 - x} + c_0 & x < 15\,,\\
	\frac{c_1}{x (x - 5)} & x \geq 15\,,
      \end{dcases}
  \end{equation*}
  for constants $c_0$, $c_1$ that will be chosen as follows:
  By~\eqref{e:x2n} we can always choose $c_1 > 0$ so that $n_0 \leq Z$ for all $x > 15$.
  A direct calculation shows that if
  \begin{equation*}
    c_1 > 900 e^5
  \end{equation*}
  then the corner condition $\partial_x^- Z(15) > \partial_x^+ Z(15)$ is satisfied.
  Making $c_1$ larger if necessary, we can ensure that there exists $c_0 \geq 0$ such that $Z$ is continuous at $x = 15$ and $n_0 \leq Z$ for all $x > 0$.

  We claim $Z$ is a supersolution to~\eqref{e:komp}--\eqref{e:noflux}.
  Clearly
  \begin{equation*}
    J(x, Z) = \frac{ c_1 \paren{x^{4} - 9 x^{3} + 15 x^{2}+c_1}} {x^{2} \paren{x - 5}^{2}} \xrightarrow{x \to \infty} c_1\,.
  \end{equation*}
  Thus we only need to verify $\partial_x J(x, Z) \leq 0$.
  For $x > 15$ we note
  \begin{equation*}
    \partial_x J(x, Z)
      = 2 Z \partial_x Z + \partial_x J_{\lin}(x, Z)
      \leq  \frac{c_1 (15 - x)}{(x - 5)^3} < 0\,,
  \end{equation*}
  where $J_{\lin}$ is defined in~\eqref{e:Jlin}.
  For $x < 15$ we compute
  \begin{equation*}
    \partial_x J(x, Z) = -2 (c_0 + e^{20-x})(e^{20-x} + 1 - x) < 0\,,
  \end{equation*}
  provided $e^{20 - x} + 1 - x > 0$.
But this condition holds, since $x<2^4<e^5<e^{20-x}$.
  Thus $Z$ is a stationary super-solution of~\eqref{e:komp}--\eqref{e:noflux}.
  By Lemma~\ref{l:comparison} this implies $n \leq Z$ for all $t \geq 0$, concluding the proof
  that $(1+x^2)n(x,t)$ is uniformly bounded on $Q$.

Now the uniform decay properties \eqref{e:x2nPtwise} and \eqref{e:x2dxnVanish} 
for $T=\infty$ follow exactly as in \cite[pp.~3849-50]{EscobedoHerreroEA98}, 
based on a finer comparison argument for $x>R$ large and classical regularity estimates.
\end{proof}

\subsection{Energy Estimates.}\label{s:energy}
We conclude this section by establishing $L^2$ energy estimates on solutions.
While such energy estimates usually play a central role in the study of parabolic problems, they are not as helpful in the present context.
Indeed, the proofs of our main results do not use~$L^2$ energy estimates, and they are only presented here for completeness.
\begin{proposition}\label{p:energy}
Let $n$ be a solution to~\eqref{e:komp}--\eqref{e:noflux} with nonnegative initial data~$n_0$ that satisfies~\eqref{e:x2n}.
Then, for any $t>s>0$, we have
\begin{align}\label{en}
  \nonumber
  \MoveEqLeft
\int_0^\infty n^2_t(x)\,dx +\int_s^t\int_0^\infty [n_\tau^2 +x ^2(\Dx n_\tau)^2]\,dx\, d\tau
\\
&\leq \int_0^\infty n^2_s(x)\,dx
  + 2 \int_s^t\int_0^\infty x n_\tau^2 \, dx \, d\tau\,.
 \end{align}
\end{proposition}
\begin{remark}
  By Proposition~\ref{p:nBdd}, the term $\int_s^t \int_0^\infty x n_\tau^2 \, dx \, d\tau$ appearing on the right can be bounded by $C(t - s)$ for some constant $C = C(n_0)$.
\end{remark}
\begin{proof}[Proof of Proposition~\ref{p:energy}]
  Let $0 < \epsilon < R < \infty$.
  Multiplying~\eqref{e:komp} by $2n$ and integrating from~$\epsilon$ to $R$ yields
  \begin{align}
    \nonumber
    \partial_t \int_\epsilon^R n^2 \,dx
      & =
    - 2\int_\epsilon^R (\Dx n) J\, dx +\brak[\Big]{ 2Jn}_\epsilon^R
    \\
    \label{e:ee1}
    & = -2\int_\eps^R [n^2+x ^2(\Dx n)^2]dx +2\int_\eps^R x n^2dx +\Gamma(t)\,,
  \end{align}
  where
  \begin{align*}
    \Gamma(t) \defeq -\brak[\Big]{ (x^2-2x)n^2}^R_{\epsilon}
      -\frac{2}{3}\brak[\Big]{ n^3}^R_{\epsilon} + 2\brak[\Big]{ Jn }^R_{\epsilon} \,.
  \end{align*}
  When~$\epsilon < 1$ and~$R > 2$ we observe
\begin{align*}
  \Gamma(t)
    & =- n^2(R, t)R(R-2)+\eps(\eps -2) n^2(\eps, t) -\frac{2}{3}n^3(R, t)+\frac{2}{3}n^3(\eps, t)+2\brak[\Big]{ Jn }^R_{\epsilon}
  \\
  & \leq 2J(R, t)n(R, t) +\frac{2}{3}n^3(\eps, t) -2J(\epsilon, t)n(\epsilon, t)\\
  & = 2J(R, t)n(R, t)-\frac{4}{3}n^3(\epsilon, t) -2\epsilon^2 n(\epsilon, t)\partial_x n(\epsilon, t)+2\epsilon (2-\epsilon)n^2(\epsilon, t)\,.
  \end{align*}
  Since~$n$ is bounded and the flux vanishes at infinity (equation~\eqref{e:noflux}) the first term on the right vanishes as $R \to \infty$.
  The second term is bounded above by~$0$.
  Using Lemma~\ref{l:oleinik}, third term vanishes as~$\epsilon \to 0$.
  The last term vanishes as~$\epsilon \to 0$, and so
  \begin{equation*}
    \lim_{\epsilon \to 0,~ R\to \infty} \Gamma(t) = 0\,.
  \end{equation*}
  Thus sending~$\epsilon \to 0$ and $R \to \infty$ in~\eqref{e:ee1} yields~\eqref{en} as claimed.
\end{proof}

\section{Finite Time Condensation.}\label{s:bec}
In this section we prove persistence (Proposition~\ref{p:persistence}) and establish the onset of photon loss
through a singularity in the slope (the second assertion in Proposition~\ref{p:formation}).
Throughout this section we assume $n_0$ is a nonnegative bounded function satisfying~\eqref{e:x2n}, and~$n$ is the unique global solution to~\eqref{e:komp}--\eqref{e:noflux} with initial data~$n_0$.

\subsection{Persistence.}

We now prove Proposition~\ref{p:persistence} and show that photon loss begins, it will never stop.
\begin{proof}[Proof of Proposition~\ref{p:persistence}]
  Suppose for some $T > 0 $ we have $n_T(0) > 0$.
  We claim that $n_t(0) > 0$ for all $t > T$.
  Once this claim is established, Proposition~\ref{p:persistence} follows immediately by setting~$t_* = \inf\set{t > 0 \st n_t(0) > 0}$.

  To prove the claim recall by Lemma~\ref{l:oleinik} we know $\partial_x n_t \geq -\varphi_t$, where~$\varphi_t$ is defined by~\eqref{e:varphiDef}.
  Integrating in~$x$ this implies
  \begin{align*}
    n_T(x) \geq \paren[\Big]{ n_T(0) - \int_0^x \varphi_T(y) \, dy }_+\,.
  \end{align*}
  Here notation $z_+$ denotes $\max\set{z, 0}$, the positive part of $z$.
  Since the function $\int_0^x \varphi_T(y) \, dy$ is convex in~$x$, we must have
  \begin{equation*}
    n_T(x) \geq (a_T -b_T x)_+ \,, 
  \end{equation*}
  where
  \begin{equation}\label{e:abID}
    a_T = n_T(0) > 0\,,
    \qquad
    b_T = \frac{a_T}{R}\,,
  \end{equation}
  and $R > 0$ is uniquely determined from
  \begin{equation*}
    a_T - \int_0^R \varphi_T(y) \, dy = 0\,.
  \end{equation*}

  Now for $t > T$, we define $a_t$ and $b_t$ to solve the ODE
  \begin{equation}\label{e:ab}
    \partial_t a=-2a(1+b)\,, \qquad \partial_t b=a-2b(1+b)\,,
  \end{equation}
  with initial data~\eqref{e:abID}.
  Let
  \begin{equation*}
    q_t(x) = (a_t - b_t x)_+\,,
  \end{equation*}
  for $t \geq T$.
  We claim that $q$ is a sub-solution to~\eqref{e:komp}--\eqref{e:noflux} (as in Definition~\ref{d:subsol}) for all $t \geq T$.
  To see this, note first that $\partial_t (b / a) = 1$ and hence
  \begin{equation}\label{e:b}
    b_t = \frac{a_t b_T}{a_T} + a_t(t - T)\,.
  \end{equation}
  Using the first equation in~\eqref{e:ab} it now follows that both
  \begin{equation*}
    a_t > 0\qquad\text{and}\qquad
    b_t > 0\,,
  \end{equation*}
  for all $t \geq T$.
  Moreover, for~$\hat x_t \defeq a_t / b_t$, equation~\eqref{e:b} implies
  \begin{equation*}
    \hat x_t = \frac{R}{1 + R(t - T)} \in (0, R)\,,
  \end{equation*}
  for all $t > T$.

  Now for $t > T$ and $x \in (0, \hat x_t)$ we compute
  \begin{align*}
    \partial_t q - \partial_x J(x, q)
      =&   \partial_t a- \partial_t b x -\partial_x [-x^2 b +(x^2 -2x)(a-bx)+(a-bx)^2]\\
      =&   \partial_t a +2a(1+b) -x(  \partial_t b +2b+2b^2-a) +3x(bx-a)\\
      =& 3x(bx- a) \leq 0. 
  \end{align*}
  For $ x > \hat x_t$, $q_t = 0$ and so $\partial_t q_t = \partial_x J(x, q_t) = 0$.
  Moreover, since $b_t > 0$ we note that the appropriate corner condition holds:
  \begin{equation*}
    \partial_x^- q_t(\hat x_t) = -b_t < 0 = \partial_x^+ q_t(\hat x_t)\,.
  \end{equation*}
  Thus~$q_t$ is a sub-solution to~\eqref{e:komp}--\eqref{e:noflux} for all $t \geq T$.
  By the comparison principle (Lemma~\ref{l:comparison}) this implies $n_t(x) \geq q_t(x)$ for all $t \geq T$ and $x \geq 0$.
  This implies $n_t(0) \geq q_t(0) = a_t > 0$ for all $t \geq T$, finishing the proof.
\end{proof}

\subsection{Onset of Photon Loss (Slope Condition).}
We now prove the second assertion in Proposition~\ref{p:formation}, which states that if $\partial_x n_0(0) > 1$, 
then photon loss must commence at or before the time $\bar t_*$ given by~\eqref{e:tStarRiccati}.
Before delving into the details of the rigorous proof, we present a quick heuristic derivation.
Let $w_t = \partial_x n_t(0)$ and differentiate equation~\eqref{e:komp} in~$x$.
Using the fact that $n_t(0) = 0$ for $t < t_*$ we formally obtain
\begin{equation}\label{e:riccati}
  \partial_t w = 2w - 4w + 2 w^2 = 2w(w-1)\,.
\end{equation}
This is a Riccati equation which can readily be integrated.
If $w_0 = \partial_x n_0(0) > 1$, then~\eqref{e:riccati} develops a singularity at time~$t_*$ given by~\eqref{e:tStarRiccati}.
To convert the above heuristic into a rigorous proof, we need to construct suitable sub and super-solutions.
\begin{proof}[Proof of the second assertion in Proposition~\ref{p:formation}]
  We will first prove~\eqref{e:tStarRiccati} holds by constructing sub-solutions using modified sideways parabolas.
  More precisely, the sub-solutions we construct will be of the form 
  \begin{equation}\label{d:nsubdef}
    z_t(x) = (u_t(x)-c_t x^2)_+ \,,
  \end{equation}
  where
  \begin{equation}\label{e:u}
    u_t(x) = \frac{\sqrt{a_t^2 + 2 b_t x} - a_t }{b_t}\,,
  \end{equation}
  and the functions $a, b, c$, will be chosen shortly, with $b > 0, c > 0$. 
  Note that $u$ is determined implicitly from the upper branch of parabolic arcs
  \begin{equation}\label{d:para1}
    x = a_t u_t(x) + \frac12{b_t }u_t(x)^2 \,.
  \end{equation}
  From this we compute
  \begin{subequations}
  \begin{gather}
    \label{e:dxU}
    1 = (a+bu)\Dx u ,
    \\
    \label{e:dx2U}
    0 = (a+bu)\Dxx u + b (\Dx u)^2,
    \\
    \label{e:dtU}
    0 = u \partial_t a  + \frac12u^2{\partial_t b}  + (a+bu)\Dt u \,.
  \end{gather}
  \end{subequations}
  When $a_t > 0$, we have $u_t(0)=0$, and hence $\partial_x z_t(0) = \Dx u_t(0)=1/a_t$.
  Our aim is to choose~$a$ so that $\partial_x z_t(0) = 1/a_t$ satisfies the Riccati equation~\eqref{e:riccati}
  until blowup.
  This boils down to letting~$a$ solve
  \begin{equation}\label{e:invslope1}
    \partial_t a_t
      = 2a_t - 2\,.
  \end{equation} 

  Notice $u\geq0$, $a+bu\ge0$, and $(a+bu)u\geq x$, hence $x\Dx u\leq u$.
  Define the (non-linear) differential operator $\mathcal L$ by
  \begin{equation*}
    \mathcal L n
      \defeq
      \partial_t n - \partial_x J(x, n)
      = \Dt n- x^2\Dxx n-\Dx n(2n+x^2)+2n(1-x) \,,
  \end{equation*}
  and compute, when $z>0$,
  \[
  \mathcal L z = \partial_t u - x^2\partial_t c- x^2(\partial_x^2 u - 2c) +(x^2+2z)(-\partial_x u+2cx)+2z(1-x),
  \]
  whence, since $cx^2=u-z$,
  \begin{align*}
    (a+bu)\mathcal L z  &= 
    (a+bu)\left( -x^2\partial_t c  + 2u +  (x^2+2z)2cx - 2zx \right)
    \\ & \qquad 
    -u\Dt a -\frac12u^2\Dt b + x^2 b(\Dx u)^2 - x^2 - 2z
    \\ &\leq (a+bu)\left(-x^2\Dt c + 2cx^3 +2zx(2c-1) \right)
    \\ &\qquad 
    + u(-\Dt a+2a-2)+ u^2\left(-\frac12 \partial_t b+ 3b\right) + x^2(-1+2c )\,.
  \end{align*}
  If
  \begin{equation}\label{e:bevol1}
    \partial_t c = c \,,
    \qquad
    0 < x < \frac{1}{2}\,,
    \qquad
    0<c < \frac{1}{2}\,,
    \qquad
    \partial_t b = 6b \,,
  \end{equation}
  then we have $\mathcal L z<0$.
  Thus, from the above we choose
  $$
    a_t = 1-(1 -a_0)e^{2t} \,,
    \qquad
    b_t = b_0 e^{6t} \,,
    \qquad\text{and}\qquad
    c_t = c_0 e^{t}\,,
  $$
  with $a_0 \in (0, 1)$, $b_0, c_0 > 0$ to be determined shortly.

  In order to ensure $z_0 \leq n_0$, pick $\epsilon > 0$ and let
  \begin{equation*}
    a_0 = \frac{1}{\partial_x n_0(0) - \epsilon} \in (0, 1)\,,
    \qquad
    t^*_\epsilon = \frac{1}{2}\abs{\ln (1-a_0)} \,.
  \end{equation*}
  Due to this choice of $t_\epsilon^*$ we have $a_{t} > 0$ for $t < t_\epsilon^*$, $a_{t_\epsilon^*} = 0$, and $a_t < 0$ for $t > t_\epsilon^*$.
  (Notice $u_t(0)>0$ for $t> t_\epsilon^*$.)
  Next choose $c_0<\frac{1}{2}(1-a_0)$ 
  so that 
  $$
    c_t<1/2 \quad\text{for } t\in (0, 2t_\epsilon^*) \,.
  $$
  This choice is made to ensure $z_t(x)$ remains a sub-solution up to time $t= 2t_\epsilon^*$.
  Next, we choose $b_0$ large enough to ensure that for some $\bar x \ll 1/2$ we have $z_0(x) = 0$ for all $x \in  (\bar x, \infty)$.
  Since $\partial_x z_0(0) = \partial_x n_0(0) - \epsilon < \partial_x n_0(0)$, by making $b_0$ larger if necessary we can also arrange $z_0 \leq n_0$.

  Now all the requirements in~\eqref{e:bevol1} are satisfied, and hence $z$ is a sub-solution up to time $2t_\epsilon^*$.
  Since $z_0 \leq n_0$, the comparison principle (Lemma~\ref{l:comparison}) implies $z_t \leq n_t$ for all $t \leq 2t_\epsilon^*$. 
  Using~\eqref{e:u} this implies
  \begin{equation*}
    n_t(0) \geq z_t(0) = u_t(0) \geq \frac{-2a_t}{b_t} > 0
    \qquad\text{for all } t \in (t_\epsilon^*, 2t_\epsilon^*]\,.
  \end{equation*}
  Sending~$\epsilon \to 0$ we see that
  \begin{equation*}
    t_* \leq \lim_{\epsilon \to 0} t_\epsilon^*
      = \frac{1}{2} \ln\paren[\Big]{ \frac{\partial_x n_0(0)}{\partial_x n_0(0) -1} } = \bar t_*\,,
  \end{equation*}
  which proves $t_* \leq \bar t_*$ as claimed.
  \medskip

  It remains to produce initial data for which $n_t(0)$ is continuous at $t = t_*$, and $t_* = \bar t_*$, where $\bar t_*$ is defined in~\eqref{e:tStarRiccati}.
  We will do this by constructing a super-solution~$Z$ such that $Z_t(0) = 0$ for~$0\le t \leq \bar t_*$ and $Z_t(0)>0$ for $t_*<t<\hat t$,
  with $Z_t(0)$ a continuous function of time. Moreover we can make $\Dx Z_0(0)>1$ be arbitrary.

  Once we construct $Z$, we choose $n_0$ to be any function for which $0\leq n_0 \leq Z_0$ and $\partial_x n_0(0) = \partial_x Z_0(0)$.
  For the corresponding solution, $n$, we must have
  \begin{equation*}
    0 \leq n_t(0) \leq Z_t(0) = 0 \qquad \text{for all } t \in [0, \bar t_*]\,.
  \end{equation*}
  This forces $t_* \geq \bar t_*$.
  Since we have already proved $t_* \leq \bar t_*$, this implies $t_* = \bar t_*$ as desired.
  Continuity of $n_t(0)$ at $t = t_*$ follows because $0 \leq n_t(0) \leq Z_t(0)$ and $Z_t(0)$ is continuous
  with $Z_{\bar t_*}(0) = 0$.

  It remains to construct the super-solution~$Z$.
  We will do this by choosing
  \begin{equation*}
    Z_t(x) = \begin{cases}  u_t(x)\varmin S_\gamma(x) \,, & 0<x\leq \bar x_0,
    \\ S_\gamma(x)\,, & x>\bar x_0,
    \end{cases}
  \end{equation*}
  where $S_\gamma$ is the stationary super-solution in~\eqref{d:S}, $\gamma > 0$ and $\bar x_0>0$ will be chosen shortly, 
  and~$u$ is given explicitly by~\eqref{e:u} (or implicitly from the upper branch of the parabolic arc~\eqref{d:para1}).

  As before we let~$a$ satisfy~\eqref{e:invslope1}, with $a_0\in(0,1)$ specified arbitrarily.
  One easily checks that
  \[
  |a_t|\le a_0 \quad\mbox{ for } \quad 0\le t\le \hat t:=\frac12\log \frac{1+a_0}{1-a_0} = \bar t_*+\frac12\log(1+a_0)\,.
  \]
  In this case it suffices to chose $b > 0$ to be constant in time.
  Using~\eqref{e:u}, \eqref{e:dxU}--\eqref{e:dtU}, and~\eqref{e:invslope1} we compute  
  \begin{align*}
    (a + bu) \mathcal L u
      &= 2u(1-a) + \frac{x^2 b}{(a + bu)^2} - (2u + x^2) + 2(1-x) u (a + bu)
    \\
      &=  x^2 \paren[\Big]{\frac{b}{a^2 + 2bx} - 1}  -2ax u +2 (1-x) b u^2\\
      & =  x^2 \paren[\Big]{\frac{b}{a^2 + 2bx} - 3 } + (2-x) b u^2,
  \end{align*}
  where we used~\eqref{d:para1} to obtain the last equality.
  Hence
  \begin{equation*}
    (a + bu) \mathcal L u
      \geq x^2 \paren[\Big]{\frac{b}{a_0^2+2bx} - 3} +  (2-x) b u^2.
  \end{equation*}
  Choosing~$\bar x_0 <  1/6$ and $b$ large we see that $\mathcal L u \geq 0$ for all $x \in (0, \bar x_0)$.

  Since $\partial_x u_0(0) > 1 = \partial_x \hat n_0(0)$, we can make $\bar x_0$ smaller if necessary to ensure $u_0(\bar x_0) > \hat n_0(\bar x_0)$.
  Then we can find a sufficiently small~$\gamma > 0$ for which $u_0(\bar x_0) > S_\gamma(\bar x_0)$.
  From~\eqref{e:dxU} and~\eqref{e:dtU} we see that the function~$u$ is increasing in both~$x$ and~$t$.
  Hence the functions $u_t(\cdot)$ and $S_\gamma(\cdot)$ must meet at some~$\bar x_t < \bar x_0$ where the function~$Z$ will satisfy the corner condition~\eqref{e:cornerSupSol}.
  This shows~$Z$ is a super-solution, concluding the proof.
\end{proof}

\section{Long Time Behavior.}\label{s:largetime}
This section is devoted to studying the long time convergence of solutions (Theorem~\ref{t:lim}, and the other results stated in Section~\ref{s:longtime}).
Following the convention from the previous section, we assume~$n$ is the unique global solution to~\eqref{e:komp}--\eqref{e:noflux} with initial data~$n_0$.
Recall Theorem~\ref{t:exist} guarantees that $n \in C^\infty( (0, \infty)^2 )$.

\subsection{Entropy Decay and Steady States (Lemma~\ref{ent}).}
The main goal of this section is to prove the entropy decay stated in Lemma~\ref{ent}.
We begin with a formal argument showing that the quantum entropy~$H$ defined in~\eqref{e:HdefIntro} is dissipated (see also~\cites{CaflischLevermore86,LevermoreLiuEA16}).
Note that the flux~$J$ can be rewritten as
\begin{align}\label{e:Jh}
  J&=n(n+x^2) \partial_x h(x, n)\,,
\end{align}
where
\begin{equation*}
  h(x, n)
    \defeq  x+ \ln n -\ln (n+x^2)
    =  x - \ln\paren[\Big]{ 1 + \frac{x^2}{n} } \,.
\end{equation*}
Multiplying~\eqref{e:komp} by $h(x, n)$ and integrating by parts formally gives
$$
\int_0^\infty \partial_t n \, h(x, n) \, dx=-\int_0^\infty n(n+x^2)|\partial_x h(x, n)|^2 \, dx \,.
$$
Here we assumed that the boundary term~$Jh$ vanishes both at zero and infinity.
Since the left hand side can be recast as a time derivative, this yields the dissipation relation~\eqref{e:ent}.
For convenience we rewrite~\eqref{e:ent} as
\begin{equation}\label{e:dtH}
  \partial_t H + D=0, 
\end{equation}
where the quantum entropy functional, $H = H(n)$, and the dissipation term, $D = D(n)$ can be rewritten as
\begin{align}
  \label{e:HPhi}H(n) & \defeq \int_0^\infty [xn +\Phi(x, n)]dx,
  \\
  \nonumber
  \Phi(x, n) & \defeq n\ln n -(n+x^2)\ln(n+x^2)+x^2 \ln (x^2)
    \\
    \nonumber
    &= -n\ln \paren[\Big]{1 + \frac{x^2}{n}}
      -x^2 \ln \paren[\Big]{ 1 + \frac{n}{x^2} }\,,
  \\
  \nonumber
  \llap{\text{and}\qquad} D(n) &\defeq \int_0^\infty n(n+x^2)|\partial_x h(x, n)|^2 \, dx \, .
\end{align}

In order to justify~\eqref{e:dtH} we need to ensure~$n > 0$ (so that~$D$ is defined), and show that $J h$ vanishes at both zero and infinity.
We do each of these below.

\begin{lemma}\label{l:nPositive}
If $n_0$ is not identically~$0$, then $n_t(x) >0$ for every $x>0$, $t>0$.
\end{lemma}
\begin{proof} 
  For any $\delta >0$, the equation
  $$
  \partial_t n =x^2 \partial_x^2 n +(x^2+2n)\partial_x n +2(x-1)n, \quad x\in (\delta, R), \quad t>0,
  $$
  is strictly parabolic, and the zeroth order coefficient is bounded from below.
  Thus by the strong minimum principle (see for instance~\cite[\S7.1.4]{Evans98}), $n_t(x) > 0$ for all $t > 0$, $x > \delta$.
  Sending $\delta \to 0$ finishes the proof.
\end{proof}

For the remainder of this section we will assume that the initial data is not identically~$0$, and hence the solution is strictly positive on $(0, \infty)^2$.
Next, to show that~\eqref{e:dtH} holds we need to show $Jh$ vanishes both at~$0$ and at infinity.
We use an averaging argument near $x=0$ (similar to what was used in the proof of Proposition~\ref{p:lossFormula}).
Unfortunately, as $x \to \infty$, our existence results \emph{do not} provide enough decay to guarantee that~$J h$ vanishes.
However, if $n_0(x)$ decays fast enough, then the comparison principle and our super-solutions~\eqref{d:S} provide enough decay to show that $J h$ vanishes as $x \to \infty$.
This is what we use to rigorously prove~\eqref{e:dtH}.

\begin{proof}[Proof of Lemma~\ref{ent}]
  First note that since $n_0(x) \leq C_0(1 + x^2) e^{-x}$, there must exist~$\gamma > 0$ such that $n_0 \leq S_\gamma$.
  (Recall $S_\gamma$ is the stationary super-solution defined in~\eqref{d:S}.)
  Thus using the comparison principle (Corollary~\ref{c:Sbound}) we must have $n_t(x) \leq S_\gamma(x)$ for all $t > 0$ and $x \geq 0$.

  Now fix $0 < \epsilon < R$, and define
  \begin{equation*}
    \zeta(x)=
      \begin{dcases}
	0 & 0\leq x\leq \epsilon\,, \\
	\tfrac{x}{\epsilon} -1 & \epsilon \leq x\leq 2\epsilon\,,\\
	1 & 2\epsilon \leq x\leq R\,, \\
	R+1-x, & R \leq x \leq R+1,\\
	0 & x\geq R+1\,.
      \end{dcases}
  \end{equation*}
  Multiplying~\eqref{e:komp} by $h(x, n)\zeta(x)$ and integrating by parts gives
  \begin{align}
    \nonumber
    \int_s^t \int_{\epsilon}^{R+1} \partial_t n \, h \, \zeta(x) \, dx \, d\tau
      & =-\int_s^t  \int_{\epsilon}^{R+1}  n(n+x^2)|\partial_x h(x, n)|^2 \zeta(x) \, dx
    \\
    \label{en1}
      & \qquad - \int_s^t \avint_{\epsilon}^{2\epsilon} J \, h \, dx \, d\tau
	+ \int_s^t \int_{R}^{R+1} J \, h \, dx \, d\tau \,.
  \end{align}
Note that $h \partial_t n =\partial_t (xn+\Phi)$, hence the left hand side of the above reduces to
\begin{equation*}
  \int_{\epsilon}^{R+1}(x n+ \Phi) \zeta(x)dx \Big|_s^{t} \,.
\end{equation*}
For the right hand side we note
\begin{equation}\label{e:dPhi}
\partial_n \Phi=\ln \paren[\Big]{\frac{n}{n+x^2}}
\qquad\text{and}\qquad
\partial_x \Phi= 2x \ln \paren[\Big]{\frac{x^2}{n+x^2}} \,.
\end{equation}
Thus we can regroup terms in $Jh$ to obtain
\begin{align*}
  Jh & =(x^2\partial_x n +n^2 +(2x-x^2)n)(x +\partial_n \Phi) \\
  & =\partial_x (x^2(xn +\Phi)) -3x^2 n -2x\Phi -x^2 \partial_x \Phi +(n+2x-x^2)(xn + n \partial_n \Phi)\\
  &=\partial_x (x^2(xn +\Phi)) -B \,,
\end{align*}
where
\begin{equation*}
B= xn(x^2+x-n)+2x\Phi +x^2 \partial_x \Phi +n(x^2-2x-n)\partial_n \Phi\,.
\end{equation*}

For $\epsilon$ small and $x \in [\epsilon , 2\epsilon]$,  we have
\begin{equation*}
  0< n \leq S_\gamma \leq \gamma + 2\epsilon \leq \gamma + 2
  \qquad\text{and}\qquad
  |\Phi|\leq C \,,
\end{equation*}
for some finite constant $C$.
We will subsequently allow $C$ to increase from line to line, provided it does not depend on $\epsilon$ or $R$.
Note also
\begin{equation}\label{e:xdxphi}
x|\partial_x \Phi| =2x^2\ln \left( 1+\frac{n}{x^2}\right) \leq 2n
\end{equation}
and 
$$
n|\partial_n \Phi|= n\ln \left( 1+\frac{x^2}{n}\right)\leq x^2. 
$$
These combined ensure $|x^2(xn+\Phi)|\leq C\epsilon^2$ and
$$
|B| \leq xn|x^2+x-n| +2x|\Phi| +2xn +(n+2x-x^2)x^2 \leq C\epsilon.
$$
Hence 
\begin{align}
  \nonumber
\abs[\Big]{-\int_s^t \avint_{\epsilon}^{2\epsilon} J\, h \, dx \, d\tau }
  &= \abs[\Big]{ -\frac{1}{\epsilon}\int_s^t \brak[\big]{x^2(xn+\Phi)}_\epsilon^{2\epsilon} \, d\tau 
+\int_s^t \avint_{\epsilon}^{2\epsilon}B \, dx \, d\tau }
  \\
  \label{e:Jh0}
  &\leq C(t-s)\epsilon.
\end{align}

We now bound the last term in~\eqref{en1}.
Note
$$
\int_s^t \int_{R}^{R+1} J \, h \, dx \, d\tau
  = \int_s^t \brak[\Big]{x^2(xn+\Phi)}_R^{R+1} \, d\tau +\int_s^t \int_{R}^{R+1}B \, dx \, d\tau \, .
$$
For $x \in [R,  R+1]$, we have
\begin{equation*}
  n_t(x) \leq S_\gamma(x)
    = \frac{x^2 e^{-x}}{ (1-e^{-x})^{2} } (\gamma+1-e^{-x}) \,,
\end{equation*}
and hence
$$
  n_t(x) \leq C R^2 e^{-R}\,, \qquad
$$
for all sufficiently large~$R$ and $x \in [R, R+1]$.
Therefore, since $n\mapsto|n\ln n|$ is increasing for $0<n<e^{-1}$, we find
\begin{align*}
  |\Phi| &= n \ln \paren[\Big]{ \frac{n + x^2}{n} } + x^2 \ln \paren[\Big]{ \frac{n + x^2}{x^2} }
    \leq C R^3 e^{-R}\,.
\end{align*}
As before we still use $x|\partial_x \Phi|\leq 2n$ (inequality~\eqref{e:xdxphi}), but we bound $n\partial_n \Phi$ differently.
Namely,
$$
 n|\partial_n \Phi|= n \ln\paren[\Big]{ 1 + \frac{x^2}{n} }
  \leq C R^3 e^{-R}\,.
$$
This ensures
$$
|x^2(xn+\Phi)|\leq C R^5 e^{-R}\,,
$$ 
and 
$$
|B| \leq xn(x^2+x) +2x|\Phi| +2x n +x^2n|\partial_n \Phi|  \leq C R^5 e^{-R}\,.
$$
Hence, 
\begin{equation}\label{e:Jhinf}
  \abs[\Big]{ \int_s^t \int_{R}^{R+1} J\, h \, dx \, d\tau }
    \leq C(t-s)R^5 e^{-R}
    \xrightarrow{R \to \infty} 0\,.
\end{equation}
\medskip

Finally sending~$\epsilon \to 0$ and $R \to \infty$~\eqref{en1} implies
\begin{equation}\label{eq:ent}
  H(n(\cdot, t))=H(n(\cdot, s)) - \int_s^t D(n(\cdot, \tau))\,d\tau\,.
\end{equation}
Since $n$ is smooth for $x > 0$, $t > 0$ this implies~\eqref{e:dtH} as desired.
\end{proof}
\begin{remark}
  In the previous we used the decay assumption on $n_0$ to ensure $n \leq S_\gamma$.
  We used $n \leq S_\gamma$ to obtain both the vanishing of $Jh$ near both zero (equation~\eqref{e:Jh0}), and at infinity (equation~\eqref{e:Jhinf}).
  The use of $n \leq S_\gamma$ to show that $J h$ vanishes at $0$ can be avoided by using
  Proposition~\ref{p:nBdd}
  instead.
  We have so far not managed to avoid the use of $n \leq S_\gamma$ to show that $J h$ vanishes at infinity.
\end{remark}

The entropy dissipation lemma suggests that the long time limit of solutions to~\eqref{e:komp}--\eqref{e:noflux} is an equilibrium solution for which
$$
J(x, n)=n(n+x^2)\partial_x h(x, n)=0 \,.
$$
This equation can be directly solved, and the nonnegative solutions are precisely the Bose--Einstein equilibria~\eqref{e:fhat} (equivalently~\eqref{e:nhat}).

\begin{lemma}\label{l:stationary}
  Let $b > 0$ be a~$C^1$ function on $[0, \infty)$.
  Then $D(b) = 0$ if and only if $J(b) = 0$ if and only if there exists $\mu \in [0, \infty)$ such that $b = \hat n_\mu$.
\end{lemma}
We remark that the stationary solutions $\hat n_\mu$ can also be characterized as minimizers of the quantum entropy functional,
but we do not need this characterization in our proofs.
See \cite{EscobedoMischlerEA05} for a precise analysis of such extrema.
\begin{proof}[Proof of Lemma~\ref{l:stationary}]
  Clearly $D(b) = 0$ if and only if~$J(b) = 0$.
  Also~$J(b) = 0$ if and only if $h(x, b) = -\mu$ for some constant~$\mu \in \R$.
  That is,
  \begin{equation*}
    x + \ln \paren[\Big]{\frac{b}{b + x^2}} = -\mu\,.
  \end{equation*}
  Solving for~$b$ and using the fact that~$b > 0$ implies $\mu \in [0, \infty)$ and $b = \hat n_\mu$.
\end{proof}

\subsection{The \texorpdfstring{$\omega$}{omega}-limit Set.}

Our aim in this section is to show that the $\omega$-limit set of any trajectory is non-empty, and invariant under the dynamics.
It will be convenient to let~$U_t$ be the solution operator of~\eqref{e:komp}--\eqref{e:noflux}.
That is, given any nonnegative initial data $a$ satisfying~\eqref{e:expDecay}, let
\[
  U_t a \defeq n_t\,,
\]
where~$n$ is the solution of \eqref{e:komp}--\eqref{e:noflux} with initial data $n_0 = a$.
Note, we assumed the faster decay~\eqref{e:expDecay} on the initial data (as opposed to the slower decay~\eqref{e:x2n} that is required to prove existence).
Under the assumption~\eqref{e:expDecay} there exists $\gamma > 0$ such that $a \leq S_\gamma$, and hence by Corollary~\ref{c:Sbound} we must have $U_t a \leq S_\gamma$ for all $t \geq 0$.
Thus~$\norm{U_t a}_{L^\infty} \leq \norm{S_\gamma}_{L^\infty} \leq 1 + 2\gamma$.
Define the set~$A_\gamma$ by
\begin{gather*}
  A_\gamma \defeq \set{ a\in L^\infty(0,\infty) \st 0 \leq a(x) \leq S_\gamma
    \quad\text{for } x\in(0,\infty) }\,,
\end{gather*}
and note that~$A_\gamma$ is positively invariant under the semi-flow induced by the solution operator.
That is,
\[
  U_t A_\gamma \subseteq A_\gamma \,, \qquad \text{for all } t\geq 0\,.
\]
Given any~$a \in A_\gamma$, recall the usual $\omega$-limit is defined by
$$
\omega(a) = \bigcap_{s>0} \overline{
\set{U_t a \st  t\geq s} }\,,
$$
where the over-line notation above denotes the closure in $L^1$.
That is, $b\in\omega(a)$ if and only if there is a sequence of times $t_k\to \infty$ such that $\norm{U_{t_k}a- b}_{L^1} \to 0$.

The main purpose of this section is to show that the~$\omega$-limit set is non-empty.
\begin{proposition}[The $\omega$-limit set]\label{p:omegalimit}
  For every $a\in A_\gamma$,
Then $\omega(a)$ is not empty, and is invariant under $U(t)$, with
\begin{equation}\label{uo}
U_t \paren[\big]{ \omega(a) } = \omega(a), \qquad\text{for all }  t>0\,.
\end{equation}
\end{proposition}

The first step in proving Proposition~\ref{p:omegalimit} is to establish boundedness in $\BV$,  the space of bounded variation functions.
\begin{lemma}\label{l:BV}
  For every $a \in A_\gamma$, and any increasing sequence of times $0 < t_1 \leq t_2 \dots$ that diverges to infinity, the family $\set{ U_{t_k} a \st k \in \N }$ is bounded in $\BV$.
\end{lemma}
\begin{proof}
  By Lemma~\ref{l:oleinik} we have~\eqref{e:dxnLower} for every $x > 0$ and $t \in (0, T]$.
  By Corollary~\ref{c:Sbound} we know~$\norm{n_t}_{L^\infty} \leq \norm{S_\gamma}_{L^\infty}$ which is independent of~$t$.
  Thus choosing
  \begin{equation*}
    \alpha \geq \sqrt{6 \norm{S_\gamma}_{L^\infty} + 1} - 1\,,
  \end{equation*}
  implies the first inequality in~\eqref{e:dxnLower} holds for all $x > 0$ and $t > 0$.
  This gives
  $$
    \abs{\partial_x n_t(x)} \leq
      \partial_x n_t(x) +\frac{1}{t}+5x+\alpha \,,
      \quad\text{for all}\quad
      x>0,~ t > 0\,.
  $$
  Now by~\eqref{e:x2dxnVanish} and Proposition~\ref{p:nBdd} there exists~$R \geq \frac{1}{\sqrt{t_1}}$ such that
  \begin{equation*}
    \abs{\partial_x n_t(x)} \leq \frac{1}{x^2}\,,
    \qquad\text{for all } x \geq R,~ t \geq t_1\,.
  \end{equation*}
  Thus for every~$k \in \N$ we have
  \begin{align*}
    \int_0^\infty \abs{\partial_x n_{t_k}(x)} \, dx
      &\leq
	\int_0^R \paren[\Big]{ \partial_x n_{t_k} + \frac{1}{t_k}+5x+\alpha } \, dx
	+\int_R^\infty \frac{dx}{x^2}
      \\
      & \leq S_\gamma(R) +\frac{R}{t_1} +\frac{5 R^2}{2} + \alpha R + \frac1R\,, 
   \end{align*}
   concluding the proof.
\end{proof}

\begin{proof}[Proof of Proposition~\ref{p:omegalimit}]
  By Lemma~\ref{l:BV} and Helley selection principle there exists an increasing sequence of times $(t_k) \to \infty$ such that $(U_{t_k} a)$ converges in~$L^1$.
  By definition the limit must belong to~$\omega(a)$ and hence~$\omega(A_\gamma)$ is non-empty.

  To prove \eqref{uo}, choose any $b\in \omega(a)$.
  There must exist a sequence of times $(t_k) \to \infty$ such that
  $$
  \|U_{t_k} a-b\|_{L^1} \xrightarrow{k \to \infty} 0\,.
  $$
  By Lemma~\ref{l:l1contract} this implies
  \begin{equation*}
    \norm{U_{t_k + t} a-U_t b}_{L^1}
    \leq \norm{U_{t_k} a- b}_{L^1}
    \xrightarrow{k \to \infty} 0\,,
  \end{equation*}
  and hence $U_t b \in \omega(a)$.
  This shows $U_t( \omega(a) ) \subseteq \omega(a)$.

  For the reverse inclusion, choose any~$\hat b \in U_t( \omega(a) )$.
  By definition, there exists~$b^* \in \omega(a)$ such that $U_t b^* = \hat b$.
  For this~$b^*$ there exists a sequence~$t_k \to \infty$ such that $(U_{t_k} a) \to b^*$ in~$L^1$.
  Hence, by Lemma~\ref{l:l1contract} we have
  \begin{equation*}
    \norm{U_{t + t_k} a - \hat b}_{L^1} 
    = \norm{U_{t + t_k} a - U_t b^*}_{L^1} 
    \leq \norm{U_{t_k} a - b^*}_{L^1} \xrightarrow{k \to \infty} 0\,,
  \end{equation*}
  which shows $\hat b \in \omega(a)$.
  Thus $\omega(a) \subseteq U_t( \omega(a) )$, finishing the proof of~\eqref{uo}.
\end{proof}

\subsection{LaSalle\texorpdfstring{'}{}s Invariance Principle (Theorem~\ref{t:lim}).}

The long time convergence of solutions (Theorem~\ref{t:lim}) can now be proved by adapting LaSalle's invariance principle to our situation.
We first prove Theorem~\ref{t:lim} assuming an exponential tail bound on the initial data, and then prove the general case using the $L^1$ contraction.

\begin{proposition}[LaSalle's invariance principle]\label{p:lasalle}
  Suppose $a \in A_\gamma$ is not identically~$0$, let $n_t =U_t a$, and set
  \begin{equation*}
    H_\infty = \lim_{t \to \infty} H(n_t)\,.
  \end{equation*}
  Then there exists a unique $\mu \in [0, \infty)$ such that 
  \begin{equation*}
    \omega(a) = \set{ \hat n_\mu}\,,
    \qquad\text{and}\qquad
    \lim_{t \to \infty} \norm{U_t a - \hat n_\mu}_{L^1} = 0\,.
  \end{equation*}
  This~$\mu$ can be uniquely determined from the relation
  \begin{equation}\label{e:Hmu}
    H(\hat n_\mu) = H_\infty\,.
  \end{equation}
\end{proposition}
\begin{proof}[Proof of Proposition~\ref{p:lasalle}]
  Choose any $b \in \omega(a)$ and a sequence~$(t_k) \to \infty$ such that $U_{t_k}(a) \to b$ in~$L^1$.
  We will show that~$b = \hat n_\mu$ for~$\mu$ given by~\eqref{e:Hmu}.
  By passing to a subsequence if necessary, we may also assume $U_{t_k}(a) \to b$ almost everywhere.
  We claim that $H(U_{t_k} a)$ also converges to $H(b)$, and hence~$H(b) = H_\infty$.
  To see this, observe
  \begin{equation*}
    \abs{H(n)} \leq x n + n + n \ln\paren[\Big]{ 1 + \frac{x^2}{n} }
  \end{equation*}
  and the right hand side is an increasing function of~$n$.
  Using this and the fact that $U_t a \leq S_\gamma$ we must have
  \begin{equation*}
    \abs{ H(U_{t_k}a) } \leq x S_\gamma + S_\gamma + S_\gamma \ln\paren[\Big]{1 + \frac{x^2}{S_\gamma} } \,,
  \end{equation*}
  which is integrable.
  Thus, by the dominated convergence theorem
  \begin{equation*}
    H(U_{t_k}(a)) \to H(b)\,,
  \end{equation*}
  and hence $H(b) = H_\infty$.

  We now claim that $b$ can not be identically~$0$.
  To see this let $\ubar{a} = a \varmin \hat n_0 \leq \hat n_0$.
  By Lemma~\ref{l:comparison} this implies~$U_t \ubar{a} \leq \hat n_0$ for all~$t \geq 0$ and hence~$U_t \ubar{a}(0) = 0$ for all $t \geq 0$.
  By~\eqref{e:loss} this means $N(U_t \ubar a) = N(\ubar a)$ for all $t \geq 0$.
  Since $\ubar a \leq a$ this implies
  \begin{equation*}
    0 < \int_0^\infty (\hat n_0 \varmin a) \, dx
     = N(\ubar a)
      = N(U_t \ubar a)
      \leq N( U_t a )
      \xrightarrow{t \to \infty}
	N(b)\,.
  \end{equation*}
  The first inequality is strict since $a$ is not identically~$0$ by assumption, and $\hat n_0 > 0$.
  Thus~$0 < N(\ubar a) \leq N(b)$ showing that~$b$ is not identically~$0$.

  We now claim there exists a~$\mu \in [0, \infty)$ such that $b = \hat n_\mu$.
  To see this recall that by Proposition~\ref{p:omegalimit} we know that $U_t(b) \in \omega(a)$ for every~$t \geq 0$.
  Hence, for every $t \geq 0$ we must also have $H(U_t(b)) = H_\infty$ and hence $\partial_t (H(U_t(b))) = 0$.
  Using~\eqref{e:dtH} this implies $D( U_t(b) ) = 0$ for every $t \geq 0$, and thus, in particular $D(b) = 0$.
  Since all solutions of~$D(b) = 0$ are of the form~\eqref{e:nhat} (Lemma~\ref{l:stationary}), there exists~$\mu \in [0, \infty)$ such that~$b =\hat n_\mu$.
  Of course, since~$H(b) = H_\infty$ the identity~\eqref{e:Hmu} holds.

  We will now show that the~$\mu$ satisfying~\eqref{e:Hmu} must be unique.
  To see this, note that
  \begin{equation*}
    \partial_\nu H(\hat n_\nu)=-\nu \int_0^\infty \frac{x^2 e^{x+\nu}}{(e^{x+\nu}-1)^2} \, dx <0, 
  \end{equation*}
  and so the function~$\nu \mapsto H(\hat n_\nu)$ is strictly decreasing.
  Thus there can be at most one solution to~\eqref{e:Hmu}, proving uniqueness.

  The above shows that for any, arbitrarily chosen, $b \in \omega(a)$, we must have~$b = \hat n_\mu$, with~$\mu$ uniquely determined by~\eqref{e:Hmu}.
  Hence~$\omega(a) = \set{\hat n_\mu}$ where~$\mu \in [0, \infty)$ is the unique number satisfying~\eqref{e:Hmu}.

  Finally, to show~$L^1$ convergence, we note that the above shows $U_{t_k} a \to \hat n_\mu$ in $L^1$ for some sequence of times with $(t_k) \to \infty$.
  For any $t > t_k$, the semi-group property, the fact that~$\hat n_\mu$ is invariant under~$U$ implies and the~$L^1$ contraction (Lemma~\ref{l:l1contract}) imply
  \begin{equation*}
    \norm{U_t a - \hat n_\mu}_{L^1}
    = \norm{U_{t-t_k} U_{t_k} a - U_{t - t_k} \hat n_\mu}_{L^1}
    \leq \norm{U_{t_k} a - \hat n_\mu}_{L^1} \,.
  \end{equation*}
  This immediately implies~$U_t a \to \hat n_\mu$ in $L^1$ as~$t \to \infty$.
\end{proof}

We are now in a position to prove Theorem~\ref{t:lim}.
The main idea in the proof is to truncate the initial data to $[0, R]$, and use Proposition~\ref{p:lasalle} to obtain the long time limit solutions with the truncated initial data.
Next we use the $L^1$ contraction (Lemma~\ref{l:l1contract}) to send $R \to \infty$.

\begin{proof}[Proof of Theorem~\ref{t:lim}]
  For any~$R > 0$ we write~$n_0 = a_R + b_R$ where~$a_R = \one_{[0, R]} n_0$.
  Since~$a_R(x) = 0$ for all $x > R$ there must exist~$\gamma = \gamma(R) \geq 0$ such that~$a_R \in A_\gamma$ for some~$\gamma \geq 0$.
  Since $a_R \in A_\gamma$, Proposition~\ref{p:lasalle} applies and there exists a unique~$\nu(R)$ such that $U_t a_R \to \hat n_{\nu(R)}$ in~$L^1$ as $t \to \infty$.

Notice that we have $a_R\leq a_{R'}$ whenever $0<R<R'$. 
Due to the comparison principle it follows $U_ta_R \leq U_t a_{R'}$ for all $t>0$,
whence $\hat n_{\nu(R)}\le \hat n_{\nu(R')}$.
This implies ${\nu(R)}\geq {\nu(R')}$
since the function ~$\mu \mapsto \hat n_\mu(x)$ is decreasing as a function of~$\mu$
for every~$x>0$. 
Thus the limit $\mu \defeq \lim_{R\to\infty}\nu(R)$ exists in $[0,\infty)$. 
Moreover $\|\hat n_{\nu(R)}-\hat n_\mu\|_{L^1}\to0$ as $R\to\infty$ 
by dominated convergence.

We now infer that $n_t\to \hat n_\mu$ in $L^1$ as $t\to\infty$,
by a standard triangle argument:
Given $\epsilon>0$, we may choose $R$ sufficiently large so 
both $\|\hat n_{\nu(R)}-\hat n_\mu\|_{L^1}$
and $\|n_0-a_R\|_{L^1}$ are less than $\frac\epsilon 3$.
Then by the $L^1$ contraction property,
  \begin{align*}
    \norm{n_t - \hat n_\mu}_{L^1} 
      &\leq \norm{n_t - U_t a_R}_{L^1}
	+ \norm{U_t a_R - \hat n_{\nu(R)}}_{L^1}
	+ \norm{\hat n_{\nu(R)} - \hat n_{\mu}}_{L^1}
    \\
      &\leq \norm{n_0 - a_R}_{L^1}
	+ \norm{U_t a_R - \hat n_{\nu(R)}}_{L^1}
	+ \norm{\hat n_{\nu(R)} - \hat n_{\mu}}_{L^1}
	<\epsilon 
  \end{align*}
for sufficiently large $t$.
This implies~\eqref{e:l1conv} as claimed.

  Finally, the identity~\eqref{e:muDef} follows immediately from~\eqref{e:l1conv} and Proposition~\ref{p:lossFormula}.
  Indeed, by~\eqref{e:loss} we have
  \begin{equation*}
    N(n_t) = N(n_0) - \int_0^t n_s(0)^2 \, ds\,.
  \end{equation*}
  By~\eqref{e:l1conv} the left hand side converges to~$N(\hat n_\mu)$ as~$t \to \infty$, yielding~\eqref{e:muDef} as claimed.
\end{proof}

\subsection{Convergence Rate.}

We will now use the entropy to bound the rate at which solutions converge to equilibrium.

\begin{proof}[Proof of Proposition~\ref{p:rate}]
Let $b=\hat n_\mu=\lim_{t\to\infty}n_t$.
Taylor expanding $\Phi$ in~$n$ we see
\begin{align}
  \nonumber
  H(n)-H(b)
    & =\int_0^\infty
      \paren[\Big]{
	(x+\partial_n \Phi(x, b))(n -b)
	+\frac{1}{2}\partial_n^2\Phi(x, \tilde n) (n-b)^2
      } \, dx
    \\
    \nonumber
    & =\int_0^\infty h(x, b) (n-b) \, dx
      +\frac{1}{2}\int_0^\infty \partial_n^2\Phi(x, \tilde n) (n-b)^2 \, dx
    \\ 
    \label{e:HnHb1}
    & = -\mu (N(n) -N(b))
      +\frac{1}{2}\int_0^\infty \partial_n^2\Phi(x, \tilde n) (n-b)^2 \, dx
      \,,
\end{align}
where $\tilde n$ is some intermediate value between $n$ and $b$.

Since~$n_0$ satisfies~\eqref{e:expDecay} we can choose~$\gamma$ large so that~$n_0 \leq S_\gamma$.
Thus by Corollary~\ref{c:Sbound} we must have~$n_t \leq S_\gamma$ for all $t \geq 0$.
This implies
$$
  \partial_n^2\Phi(x, \tilde n)=\frac{x^2}{\tilde n (\tilde n+x^2)}
    \geq \frac{x^2}{S_\gamma (S_\gamma +x^2)},
$$
and hence~\eqref{e:HnHb1} and~\eqref{e:loss} imply
\begin{equation}\label{e:SgammaWeightedRate}
  \int_0^\infty \frac{x^2}{S_\gamma ( S_\gamma + x^2) }(n_t-b)^2 \, dx
    \leq H(n_t)-H(b)+\mu \int_t^\infty n_s(0)^2 \, ds \,.
\end{equation}

Using the Cauchy--Schwarz inequality and~\eqref{e:SgammaWeightedRate} we have 
\begin{align*}
  \norm{ x^2 (n_t - b) }^2_{L^1}
    &\leq \paren[\Big]{
	\int_0^\infty \frac{x^2}{S_\gamma(S_\gamma + x^2)} (n_t - b)^2 \, dx
      }
      \paren[\Big]{
	\int_0^\infty S_\gamma(S_\gamma + x^2) \, dx
      }
  \\
  & \leq C \paren[\Big]{ H(n_t)-H(b)+\mu \int_t^\infty n_s(0)^2 \, ds } \,.
\end{align*}
Here $C = C(\gamma) = \int_0^\infty S_\gamma (S_\gamma + x^2) \, dx$.
Since this is a constant that can be bounded in terms of~$C_0$ alone, the proof is complete.
\end{proof}
\begin{remark}
  Another estimate on the rate of convergence is~\eqref{e:SgammaWeightedRate}.
  This bounds the norm of~$n_t - b$ in the weighted~$L^2$ space with the exponentially growing weight~$x^2 / (S_\gamma (S_\gamma + x^2))$.
\end{remark}

\subsection{Mass Condition for Photon Loss, and Determining \texorpdfstring{$\mu$}{mu}.}

Finally we conclude the paper with the proofs of the first assertion in Proposition~\ref{p:formation},  and Corollary~\ref{c:mulim}.

\begin{proof}[Proof of the first assertion in Proposition~\ref{p:formation}]
  Let~$\mu$ be as in Theorem~\ref{t:lim}.
  Since $N(n_0) > N(\hat n_0) \geq N(\hat n_\mu)$, by~\eqref{e:muDef} we have
  \begin{equation*}
    \int_0^\infty n_t(0)^2 \, dt = N(n_0) - N(\hat n_\mu)
      \geq N(n_0) - N(\hat n_0)
      > 0 \,.
  \end{equation*}
  Thus there must exist some $t < \infty$ for which $n_t(0) > 0$.
  This implies~$t_* < \infty$ concluding the proof.
\end{proof}

\begin{proof}[Proof of Corollary~\ref{c:mulim}]
  For the first assertion, we assume~$n_0 \geq \hat n_0$.
  By the comparison principle (Lemma~\ref{l:comparison}) this implies $n_t \geq \hat n_0$ for all $t \geq 0$.
  By Theorem~\ref{t:lim}, we also know $(n_t) \to \hat n_\mu$ in~$L^1$ as $t \to \infty$.
  However, for any~$\mu > 0$, $\hat n_\mu < \hat n_0$.
  Thus the only way we can have $(n_t) \to \hat n_\mu$ as $t \to \infty$ is if~$\mu = 0$.
  This proves the first assertion.

  For the second assertion, we again let~$\mu$ be as in Theorem~\ref{t:lim}.
  The comparison principle implies~$n_t \leq \hat n_0$ for all $t \geq 0$, and hence~$n_t(0) = 0$ for all $t \geq 0$.
  Using~\eqref{e:muDef} this implies
  \begin{equation*}
    0 = \int_0^\infty n_t(0)^2 \, dt = N(n_0) - N(\hat n_\mu)\,,
  \end{equation*}
  proving~$N(n_0) = N(\hat n_\mu)$ as claimed.

  Finally, for the third assertion let $\ubar n_0 = n_0 \varmin \hat n_0$, and let $\ubar n$ be the solution to~\eqref{e:komp}--\eqref{e:noflux} with initial data $\ubar n$.
  By the comparison principle (Lemma~\ref{l:comparison}), $n_t \geq \ubar n_t$ for all $t \geq 0$.
  By the previous assertion, $N(\ubar n_t) = N(\ubar n_0)$ for all $t \geq 0$, from which~\eqref{e:NmuLower} follows immediately.
\end{proof}

\appendix
\section{Numerical Method.}\label{s:nmethod}
We now describe the numerical method used to generate Figures~\ref{f:formationEHV} and~\ref{f:formationSlope}.
Several authors~\cites{ChangCooper70,LarsenLevermoreEA85} have introduced efficient numerical schemes for one dimensional Fokker--Planck equations which directly apply to the situation at hand.
The scheme we use is simpler to implement than the one proposed in~\cite{LarsenLevermoreEA85}, still preserves many features of the dynamics, and yields good results.
Our numerical scheme does not impose an apriori boundary condition at $x = 0$, and is not designed to not conserve the total photon number.
This is reflective of the true behavior of the equations --- solutions are unique without imposing a boundary condition at $x = 0$, and the photon number is not conserved.
However, our numerical scheme does assume a vanishing condition that is one order higher than what is guaranteed by Theorem~\ref{t:exist}.

Fix the right boundary $R > 0$, large, and a time step $\delta t > 0$ small.
Let $0 = x_0 < x_1 \dots < x_M = R$ be a (non-uniform) spatial mesh, which we choose so that the spacing $\delta x_i \defeq x_{i+1} - x_i$ is proportional to $x_i$.
The reason for this choice is because the solution decays exponentially near infinity, so not much resolution is required for large $x$.
For small $x$, however, the solution develops a jump or a cusp at $x =0$, and so higher resolution is required for small $x$.
By abuse of notation we will use $n_k$ to denote an approximation of the solution~$n$ at time $k\, \delta t$.

We split the flux $J(x,n)$ into the linear terms $J_\lin(x, n)$ (equation~\eqref{e:Jlin}) and the nonlinear term $n^2$.
Since we expect solutions to form a travelling wave towards~$0$, we use an upwind scheme~\cite{CourantIsaacsonEA52} to approximate $\partial_x (J_\lin + n^2)$.
The issue that requires some care, however, is the boundary condition at $x = 0$.
We know from Lemma~\ref{l:intJ0} that $\tau\mapsto x^2 \partial_x n_\tau$ vanishes in $L^1_{\loc}((0, \infty))$ as $x \to 0$.
If we momentarily assume the \emph{stronger} vanishing conditions
\begin{equation}\label{e:strongerVanish0}
  \lim_{x \to 0^+} x^2 \partial_x^2 n = 0\,,
  \qquad\text{and}\qquad
  \lim_{x \to 0^+} x \partial_x n = 0\,,
\end{equation}
then we must have
\begin{equation}\label{e:Jlin0}
  \partial_x J_\lin(x, n) \Bigr\rvert_{x = 0} = -2 n(0)\,,
\end{equation}
and this can readily be used as a boundary condition at $x = 0$ in a numerical scheme.
We do not presently know whether solutions to~\eqref{e:komp}--\eqref{e:noflux} satisfy the stronger vanishing condition~\eqref{e:strongerVanish0} or not.
But using~\eqref{e:Jlin0} as a boundary condition in a numerical scheme yields excellent results.

Explicitly, we discretize~\eqref{e:komp}--\eqref{e:noflux} by treating the linear terms implicitly and the non-linearity explicitly, and solve
\begin{equation}\label{e:explicitNon}
  \frac{n_{k+1} - n_k}{\delta t}
    = \partial_x^+ J^-_\lin(n_{k+1}) + \partial_x^+ n_k^2\,,
\end{equation}
at the mesh points $x_1$, \dots, $x_{M-1}$.
Here $\partial_x^+$ and $\partial_x^-$ are the left and right difference operators
\begin{equation*}
  \partial_x^+ f(x_i) = \frac{f(x_{i+1}) - f(x_i)}{x_{i+1} - x_i}
  \qquad
  \partial_x^- f(x_i) = \frac{f(x_{i}) - f(x_{i-1})}{x_{i} - x_{i-1}}\,,
\end{equation*}
and $J^-_\lin(n_{k+1})$ is an approximation of $J_\lin(n_{k+1})$ using left differences.
For boundary conditions, at the mesh points $x_0 = 0$ and $x_M = R$ we require
\begin{gather}
  \label{e:explicitBC0}
  \frac{n_{k+1}(0) - n_k(0)}{\delta t} = -2 n_{k+1}(0) + \frac{n_k^2(x_1) - n_k^2(0)}{x_1}
  \\
  \label{e:vanishInf}
  n_{k+1}(R) = 0\,.
\end{gather}

\begin{remark}\label{r:params}
  To generate Figures~\ref{f:formationEHV}--\ref{f:formationSlope}, we now solve~\eqref{e:explicitNon}--\eqref{e:vanishInf} with $M = 4000$, $R = 30$, and a mesh chosen so that $x_M - x_{M-1} \approx 0.1$.
This gives $x_1 - x_0 \approx 1.03 \times 10^{-7}$, and we choose $\delta t = x_1 - x_0$.
\end{remark}

\begin{remark*}
  One can also discretize~\eqref{e:komp}--\eqref{e:noflux} by treating the nonlinearity semi-implicitly by replacing~\eqref{e:explicitNon} and~\eqref{e:explicitBC0} by the equations
  \begin{gather}\label{e:semiimplicitNon}
    \tag{\ref*{e:explicitNon}$'$}
    \frac{n_{k+1} - n_k}{\delta t}
      = \partial_x^+ J^-_\lin(n_{k+1}) + \partial_x^+ (n_{k+1} n_k)\,,
    \\
    \tag{\ref*{e:explicitBC0}$'$}
    \frac{n_{k+1}(0) - n_k(0)}{\delta t} = -2 n_{k+1}(0) + \frac{n_{k+1}(x_1)n_k(x_1) - n_{k+1}(x_0)n_k(x_0)}{x_1 - x_0}
  \end{gather}
  While this should yield better results in theory, it is much slower in practice as the matrix defining equation~\eqref{e:semiimplicitNon} needs to reconstructed (and decomposed) at every time step.
\end{remark*}

\bibliographystyle{halpha-abbrv}
\bibliography{mainbib}
\end{document}